\documentclass[a4paper,11pt]{article}

\usepackage{a4wide}
\usepackage{graphicx}
\usepackage{amssymb,amsmath,amsthm,mathrsfs,xfrac,mathtools}
\usepackage{enumerate, enumitem}
\usepackage[title]{appendix}
\usepackage{amsfonts}
\usepackage{hyperref}
\usepackage[utf8]{inputenc}
\usepackage{color}
\usepackage{pgf,tikz}
\usetikzlibrary{arrows}
\usepackage{booktabs}
\usepackage{float}
\usepackage{ulem}

\usepackage{caption,subcaption}

\usepackage{tikz}
\newlength\fwidth
\usetikzlibrary{arrows}
\usetikzlibrary{decorations.pathreplacing}
\usetikzlibrary{plotmarks}
\usepackage{pgfplots}
\usepackage{pgfgantt}
\usepackage{pdflscape}
\pgfplotsset{compat=newest}
\pgfplotsset{plot coordinates/math parser=false}
\usepackage{varwidth}
\usepgfplotslibrary{fillbetween}
\pgfplotsset{compat = 1.3}
\usepgfplotslibrary{patchplots}
\usetikzlibrary{quotes,angles}

\usepackage[margin=10pt,font=normal,labelfont=bf,
labelsep=endash]{caption}
\usepackage{comment}

\newtheorem{proposition}{Proposition}[section]
\newtheorem{theorem}[proposition]{Theorem}
\newtheorem{lemma}[proposition]{Lemma}
\newtheorem{definition}[proposition]{Definition}
\newtheorem{corollary}[proposition]{Corollary}
\newtheorem{remark}[proposition]{Remark}
\newtheorem{lgrthm}[proposition]{Algorithm}

\numberwithin{equation}{section}
\allowdisplaybreaks[4]

\renewenvironment{proof}{\smallskip\noindent\emph{\textbf{Proof.}}%
  \hspace{1pt}}{\hspace{-5pt}{\nobreak\quad\nobreak\hfill\nobreak%
    $\square$\vspace{2pt}\par}\smallskip\goodbreak}

\newcommand{\C}[1]{\mathbf{C^{#1}}}

\renewcommand{\L}[1]{\mathbf{L^#1}}

\newcommand{\BV}{\mathbf{BV}}
\newcommand{\W}[2]{{\mathbf{W}^{#1,#2}}}
\newcommand{\modulo}[1]{{\left|#1\right|}}
\newcommand{\norma}[1]{{\left\|#1\right\|}}

\newcommand{\reali}{{\mathbb{R}}}
\newcommand{\R}{\mathbb R}

\newcommand{\N}{{\mathbb N}}
\newcommand{\interi}{{\mathbb{Z}}}

\renewcommand{\epsilon}{\varepsilon}
\renewcommand{\phi}{\varphi}

\renewcommand{\theta}{\vartheta}

\newcommand{\tv}{\mathinner{\rm TV}}

\renewcommand{\d}[1]{\mathinner{\mathrm{d}{#1}}}

\newcommand{\brho}{\boldsymbol{\rho}}

\newcommand{\dt}{{\Delta t}}
\newcommand{\dx}{{\Delta x}}
\newcommand{\dy}{{\Delta y}}
\newcommand{\rh}[1]{\rho^{c,n}_{#1}}
\newcommand{\x}[1]{x_{#1}}

\newcommand{\vs}{v^{stat}}
\newcommand{\vd}{v_c^{dyn}}

\newcommand{\del}{\partial}

\newcommand{\be}{\begin{equation}}
\newcommand{\ee}{\end{equation}}

\definecolor{ffqqqq}{rgb}{1.,0.,0.}
\definecolor{uuuuuu}{rgb}{0.26666666666666666,0.26666666666666666,0.26666666666666666}

\makeatletter
\let\@fnsymbol\@arabic
\makeatother

\makeatletter
\newcommand\appendix@section[1]{%
\refstepcounter{section}%
\orig@section*{Appendix \@Alph\c@section: #1}%
\addcontentsline{toc}{section}{Appendix \@Alph\c@section: #1}%
}
\let\orig@section\section
\g@addto@macro\appendix{\let\section\appendix@section}
\makeatother

\title{A nonlocal model for heterogeneous material flow \\on conveyor belts}

\author{Paola Goatin\footnotemark[1] \and  Simone G\"ottlich\footnotemark[2]  \and  Fabian Ziegler\footnotemark[2]}

\date{\today}

\definecolor{Green}{rgb}{0.13, 0.55, 0.13}

\begin{document}

\maketitle
\footnotetext[1]{Universit\'e C\^ote d'Azur, Inria,
  CNRS, LJAD, 2004 route des Lucioles - BP 93, 06902 Sophia Antipolis
  Cedex, France. E-mail: \texttt{paola.goatin@inria.fr}}
 \footnotetext[3]{University of Mannheim, Department of Mathematics,
  68131 Mannheim, Germany. Email: \texttt{goettlich@uni-mannheim.de
    }}

\begin{abstract}
In this paper, a ﬁnite volume approximation scheme is used to solve a nonlocal macroscopic material ﬂow model in two space dimensions, accounting for the presence of boundaries in the nonlocal terms. Based on a previous result for the scalar case, we extend the setting to a system of heterogeneous material on bounded domains. We prove the convergence of the approximate solutions constructed  using the Roe scheme with dimensional splitting. We consider a regularized version of the flux function and establish BV bounds for a Roe-type scheme. Numerical tests show a good agreement with microscopic simulations. 
\end{abstract}

\noindent

  \medskip


  \noindent\textit{2010~Mathematics Subject Classification:} 35L65, 65M12

  \medskip

\section{Introduction}
\label{sec:intro}

In this paper, we consider the Cauchy problem for a  system of nonlocal conservation laws in two space dimensions describing the movement of $N\in\N$ classes of objects of different sizes and corresponding maximal densities $R_c>0$, $c=1,\ldots,N$, on a domain $\Omega\subseteq\mathbb{R}^2$. This extends a previous result for the scalar case~\cite{Rossietal}. The initial-boundary value problem for the particle densities $\brho=(\rho^1,\dots,\rho^N)$ as a function of time $t$ and position $(x,y)\in\Omega$ is given by
\begin{equation}
  \label{eq:1}
  \left\{
    \begin{array}{l@{\qquad}r@{\,}c@{\,}l}
      \del_t \rho^c + \nabla \cdot \left(\rho^c \left( \boldsymbol{\vs} (x,y) + \boldsymbol{\vd} [\brho](x,y)\right)\right) =0,
      & (t,x,y)
      &\in
      & [0,T] \times \Omega,
      \\
      \rho^c(0,x,y) = \rho_o^c(x,y),
      & (x,y)
      &\in
      & \Omega,
    \end{array}
  \right. c=1,\ldots,N,
\end{equation}
for any $T>0$, where for weighting factors $\epsilon_c>0$ 
\begin{equation}
  \label{eq:vdyn}
  r=\sum_{c=1}^N \alpha_c\rho^c,
  \qquad
  \boldsymbol{\vd}[\boldsymbol{\rho}] = H (\tilde{\eta}_c*r - r_{\max}) \, \boldsymbol{I_c}[r]
  \qquad
    \mbox{with }
    \boldsymbol{I_c}[r]
    =
    - \epsilon_c \dfrac{\nabla(\eta_c * r)}{\sqrt{1+ \norma{\nabla(\eta_c * r)}^2}}.
  \end{equation}
  Above, $H$ denotes a smooth approximation of the Heaviside function, which becomes active whenever the maximal density $r_{\max}>0$ 
  is exceeded, {while $\tilde{\eta}_c,\eta_c$ are positive, smooth mollifiers depending on the object sizes. 
  The maximal density $r_{\max}$ is expressed in terms of the smallest object density. This concept, referred to as \texttt{PCE}, originates from traffic flow modeling, where it is employed to account for the varying sizes of vehicles on the road~\cite{van2008fastlane}. 
  %
  Smaller objects occupy less space, and thus their maximal density is higher than that of larger objects. The weighted total density $r=\sum_{c=1}^N\alpha_c\rho^c$ is defined as the sum of the individual particle densities $\rho^c$, rescaled by a weight $\alpha_c$. The weights $\alpha_c$ are determined by comparing the maximum density $r_{\max}^c$ of particle class $c$ to that of the smallest particle class. Without loss of generality, we assume that $c=1$ denotes the class of the smallest particles, so that $r_{\max}^1>r_{\max}^c$ for all $c>1$. We then set $r_{\max}=r_{\max}^1$ and $\alpha_c=r_{\max}^1/r_{\max}^c$. \\  
The scalar version of~\ref{eq:1} was introduced in~\cite{original}, to model the flow of particles on a conveyor belt. The velocity vector field induced by the conveyor belt is denoted by 
$\boldsymbol{\vs}$ and is independent of both time and the density $\boldsymbol{\rho}$, and common to all particle classes. For an illustration of the geometry of the conveyor belt, we refer to Section 4, Figure~\ref{img:staticvelocityfield}.
Below the maximal density $r_{\max}$, parts are transported with the velocity of the conveyor belt. If the density exceeds $r_{max}$, the dynamic velocity vector field $\boldsymbol{\vd}$ becomes active and models the collision of particles via the operator $\boldsymbol{I_c}[r]$.
The negative gradient of the convolution $\eta_c * r$ pushes the mass towards lower density regions. The particular choice of the collision operator $\boldsymbol{I_c}(r)$ was introduced in~\cite{ColomboGaravelloMercier2012} to describe crowd dynamics.
In particular, the norm appearing in the denominator is the Euclidean norm in $\reali^2$, so that the denominator acts as a smooth normalisation factor. We choose the Euclidian norm to reflect the numerical implementation. 

To model boundaries in the form of walls and other obstacles, we follow an approach introduced in~\cite{burger2020} and later developed in \cite{goatin2023, GoatinRossi2024}, by incorporating the information in the convolution kernel. For this, assume $\Omega^c=\mathbb{R}^2\setminus\Omega$ to be a compact set consisting of a finite number $M\in\mathbb{N}$ of connected components $\Omega^c=\Omega^c_1\cup\cdots\cup\Omega^c_M$. We then replace $r$ in the nonlocal operator by the augmented total density $r_\Omega=\sum_{c=1}^{N+1} \alpha_c\rho_\Omega^c$ with $\alpha_{N+1}=1$ and $\rho_o^c=0$ for $x\in\Omega^c, c=1,\dots,N$ and where $\rho^c_\Omega:\mathbb{R}^2\rightarrow\mathbb{R}_+$, is given by
\begin{equation}
\rho^c_\Omega:= 
                     \begin{cases}
                       \rho^c & \mbox{if } c=1,\dots,N,
                       \\
                       \sum_{l=1}^M R_{N+l}\chi_{\Omega_l^c} & \mbox{if } c=N+1,
                       \end{cases}
\end{equation}
with $R_{N+l}>0, l=1,\dots,M$, big enough so that for $\brho_{\boldsymbol{\Omega}}=(\rho_\Omega^1,\dots,\rho_\Omega^{N+1})$ it holds 
\[
\hbox{$\left(\boldsymbol{\vs} (x,y) + \boldsymbol{\vd} [\brho_{\boldsymbol{\Omega}}](x,y)\right)\cdot \mathbf{n}(x,y)\leq 0$ for all $(x,y)\in\partial\Omega$,
$t\geq0$, }
\]
with $\mathbf{n}$ being the outward normal to $\Omega$. Thus, the domain $\Omega$ is invariant and we can rewrite the Cauchy problem \eqref{eq:1} on $\R^2$:
\begin{equation}
  \label{eq:extended}
  \left\{
    \begin{array}{l@{\qquad}r@{\,}c@{\,}l}
      \del_t \rho^c + \nabla \cdot \left(\rho^c \left( \boldsymbol{\vs} (x,y) + \boldsymbol{\vd} [\brho_{\boldsymbol{\Omega}}](x,y)\right)\right) =0,
      & (t,x,y)
      &\in
      & [0,T] \times \reali^2,
      \\
      \rho^c(0,x,y) = \rho_o^c(x,y),
      & (x,y)
      &\in
      & \reali^2,
    \end{array}
  \right. c=1,\ldots,N.
\end{equation}

Conservation laws with nonlocal flux functions have been introduced in the literature to describe transport phenomena accounting for nonlocal interaction effects among agents, such as road traffic flow~\cite{BlandinGoatin2016} or pedestrian dynamics~\cite{ColomboGaravelloMercier2012}.
General well-posedness results have been provided by~\cite{ACT2015} in the scalar one-dimensional case, while~\cite{ACG2015, GoatinRossi2024} deal with systems of nonlocal conservation laws in multi-space dimensions. Even if the results in \cite{ACG2015} apply to our problem~\eqref{eq:1}, estimates in~\cite{ACG2015}  where obtained for general flux functions using finite volume approximate solutions constructed via Lax-Friedrichs scheme. In \cite{Rossietal} sharp estimates for the Roe scheme were derived, which is known to give less diffusive solutions in the case of nonlocal problems in multi-D. The aim of the present paper is to extend the results for the Roe scheme in~\cite{Rossietal} to the system case, including the presence of obstacles. This mainly involves some further assumptions on the interaction term and some changes in the estimates in Appendix \ref{sec:teclem}. {We recall that, more recently, an alternative approach to prove well-posedness of some nonlocal conservation laws, based on the method of characteristics and on Banach's  fixed-point theorem, was proposed in~\cite{KP2017,KPS2018}. In particular, no entropy condition is required to prove uniqueness and stability of solutions, showing that weak solutions are indeed unique. Nevertheless, these results 
cannot give information on the convergence of the finite volume numerical schemes studied in this paper.}

We finally remark that, even if the original model proposes the use of the discontinuous Heaviside function, the stability of the numerical schemes requires a smooth approximation of it. Together with the introduction of the mollifier $\Tilde{\eta}$, the maximal density constraint is not satisfied. Indeed, the $\L\infty$-norm of the derivatives $H'$ and $\nabla\tilde{\eta}$ are included in several estimates (for instance in the constants~\eqref{eq:K1defroe} and \eqref{eq:K2defroe} of the $\BV$-estimates), which blow up with it.

The paper is organised as follows: for a better overview, we introduce our main results in Section~\ref{sec:MR} and we postpone the proof of convergence of the approximate solutions constructed by the Roe scheme in Section~\ref{sec:exist}. This includes a proof of the Lipschitz continuous dependence on the initial data in Section~\ref{sec:lipdep}. Section~\ref{sec:results} is devoted to discussion of the numerical tests.

\section{Main results}
\label{sec:MR}

Throughout the paper, we will denote
$\mathcal{I} (r,s) := [\min\left\{r,s\right\}, \max\left\{r,s
\right\}]$, for any $r, s \in \R$.  We require the following
assumptions to hold for $c\in\{1,\dots,N\}$:
\begin{enumerate}[label={$\boldsymbol{(\Omega.1)}$}]
\item \label{Omega1}The domain $\Omega\subseteq \reali^2$ is a non-empty open set with smooth boundary $\partial\Omega$, so that the outward normal $\mathbf{n}(x,y)$ is uniquely defined for all $(x,y)\in\partial\Omega$. Moreover, we assume that $\Omega^c=\reali^2\setminus\Omega$ is a compact set consisting of a finite number $M$ of connected components.
\end{enumerate}
\begin{enumerate}[label={$\boldsymbol{(\Omega.2)}$}]
\item \label{Omega2} For all $c=1,\dots,N$ the vector fields $\boldsymbol{\vs}+\boldsymbol{v^{dyn}_c}$ point inward along the boundary $\partial\Omega$ of $\Omega$, i.e $\boldsymbol{(\vs+v^{dyn}_c)}\cdot \mathbf{n} (x,y)\leq 0$ for all $(x,y)\in\partial\Omega, t\geq0$.
\end{enumerate}
\begin{enumerate}[label={$\boldsymbol{(v)}$}]
\item \label{vs} $\boldsymbol{\vs} \in (\W2\infty\cap\C2)(\reali^2; \reali^2)$.
\end{enumerate}
\begin{enumerate}[label={$\boldsymbol{(H)}$}]
\item \label{H} The function $H$ is a smooth approximation of the
  Heaviside function, so that its derivatives $H', H''$ are bounded and $0\leq H\leq1$. In particular, we denote by $L_H$ and $L_{H'}$ the Lipschitz constants of the functions $H$ and $H'$:
  \begin{equation}
    \label{eq:L}
    L_H = \norma{H'}_{\L\infty(\reali)},
    \,L_{H'} = \norma{H''}_{\L\infty(\reali)}
  \end{equation}
\end{enumerate}
\begin{enumerate}[label={$\boldsymbol{(\eta)}$}]
\item \label{eta}
  $\tilde{\eta}_c$,$\eta_c \in (\C3 \cap \W3\infty )(\reali^2; \reali^+)$ for all $c\in\{1,\dots,N\}$.
\end{enumerate}

\begin{remark}
We note that assumption~\ref{Omega1} excludes the degenerate case of connected components $\Omega^c_i$, $i\in\{1,\ldots,M\}$, reducing to points. 
\end{remark}
\begin{remark}
    Assumptions~\ref{H} and~\ref{eta} are necessary to derive the BV estimates for the numerical discretization in Proposition~\ref{prop:bvroe}. As the regularization tends to the discontinuous case,~\eqref{eq:K1defroe}, \eqref{eq:K2defroe} blow up.
\end{remark}

Recall the definition of solution to the Cauchy problem~\eqref{eq:1},
see also~\cite{ACG2015, ACT2015, ColomboGaravelloMercier2012}.

\begin{definition}
  \label{def:sol}
 A map
  $\brho: [0,T]\to \L\infty (\Omega; \reali^N)$ is a solution
  to~\eqref{eq:1} with initial datum $\brho_o \in \L\infty (\Omega; \reali^N_+)$, if, for all $c\in\{1,\dots,N\}$, $\rho^c$ is a weak solution
  to
  \begin{equation}
    \label{eq:Kr}
    \left\{
      \begin{array}{l@{\qquad}r@{\,}c@{\,}l}
        \del_t \rho^c + \nabla \cdot \left(\rho^c V^c(t, x, y) \right) = 0,
        & (t,x,y)
        &\in
        & [0,T] \times \reali^2,
        \\
        \rho^c(0,x,y) = \rho^c_o(x,y),
        & (x,y)
        &\in
        & \reali^2,
      \end{array}
    \right.
  \end{equation}
  with 
  \begin{align*}
  V^c(t,x, y) = \boldsymbol{\vs}(x,y) - \epsilon_c 
  \, H(\tilde{\eta}_c*r_\Omega - r_{\max}) \dfrac{\nabla\left(\eta_c * r_\Omega\right)}{\sqrt{1+
      \norma{\nabla\left(\eta_c * r_\Omega\right)}^2}}(x,y),
      \end{align*}
      where $r_\Omega=\sum_{k=1}^{N+1}\alpha_k\rho_\Omega^k$, $(t,x,y)\in[0,T]\times\R^2$.
\end{definition}
\noindent Above, by weak solution, we refer
to solutions in the sense of distributions, see e.g.~\cite[Eq.~(1.6)]{Kruzkov}.

\begin{theorem}
  \label{thm:main}
  Let $\brho_o \in (\L1\cap\L\infty \cap \BV) (\Omega; \reali^N_+)$ and
  assumptions~\ref{Omega1}, \ref{Omega2}, \ref{vs}, \ref{H} and~\ref{eta} hold. Then, for all
  $T>0$, there exists a unique weak solution
  $\brho \in (\L1\cap \L\infty \cap \BV) ([0,T] \times \Omega; \reali^N_+)$ to
  problem~\eqref{eq:1}. Moreover, the following estimates hold: for
  all $t \in [0,T]$ and all $c\in\{1,...,N\}$
  \begin{align*}
    \norma{\brho(t,\cdot)}_{\L1(\Omega; \reali^N)} =  \
    & \norma{\brho_o}_{\L1(\Omega; \reali^N)},
    \\
    \norma{\rho^c(t,\cdot)}_{\L\infty(\Omega; \reali)} \leq  \
    & \norma{\rho_o^c}_{\L\infty(\Omega; \reali)} \, e^{\mathcal{C}^c_{\infty} \, t},
    \\
    \tv(\rho^c(t,\cdot)) \leq \
    & \tv(\rho_o^c) \, e^{2 \, t \, \mathcal{K}^c_1}
      + \frac{2\, \mathcal{K}^c_2}{\mathcal{K}^c_1} \left(e^{2 \, t \, \mathcal{K}^c_1} - 1\right),
    \\
    \norma{\brho(t,\cdot) - \brho(t - \tau,\cdot)}_{\L1(\Omega; \reali^N)} \leq \
    &  2 \, \mathcal{C}_t (t) \, \tau, \quad \mbox { for } \tau >0,
  \end{align*}
  where $\mathcal{C}^c_\infty$ is defined in~\eqref{eq:Cinf},
  $\mathcal{K}^c_1$ is defined in~\eqref{eq:K1defroe}, $\mathcal{K}^c_2$
  is defined in~\eqref{eq:K2defroe} and $\mathcal{C}_t(t)$ is as
  in~\eqref{eq:Ct}.
\end{theorem}

The proof of existence of solutions is based on the constructions of a converging sequence of approximate solutions and follows standard guidelines, see e.g.~\cite[Theorem~2.3]{ACG2015}. In the literature, this usually relies on Lax-Friedrichs type schemes. In this paper, we derive the necessary compactness estimates for a Roe-type scheme.
The bounds presented in Theorem~\ref{thm:main} are obtained by passing
to the limit in the corresponding discrete bounds.
Uniqueness is ensured by \cite{KPS2018} and Proposition~\ref{prop:lipdep} provides the
Lipschitz continuous dependence of solutions
to~\eqref{eq:1} on the initial data.

\begin{remark}
For nonlinear local conservation laws, weak solutions are in general not unique, and entropy conditions are used to select the admissible solution. For nonlocal transport equations, uniqueness can instead be established directly under suitable regularity assumptions on the nonlocal velocity field. In particular, Crippa and Lécureux-Mercier~\cite{CrippaMercier2011} proved well-posedness of measure solutions for multidimensional systems of continuity equations with nonlocal velocities with respect to the Wasserstein distance, while Keimer, Pflug, and Spinola~\cite{KPS2018} established existence and uniqueness of weak solutions for multidimensional nonlocal balance laws without imposing an entropy condition, using the method of characteristics and Banach's fixed-point theorem. Since the flux is linear with respect to the local variable and both the velocity and the smoothed Heaviside function are Lipschitz continuous, these results can be applied to \eqref{eq:1}. Thus, the problem admits a unique weak solution. Combined with the compactness and consistency of the Roe approximations established below, uniqueness implies that the whole approximation sequence converges to this solution.
\end{remark}

\section{Existence}
\label{sec:exist}

We introduce the uniform mesh of width $\dx$ along the $x$-axis and $\dy$
along the $y$-axis, and a time step $\dt$ subject to a CFL condition,
specified later on. For $k \in \interi$ set
\begin{align*}
  x_k = \ & (k-1/2) \dx,
  &
    y_k = \ & (k-1/2) \dy,
  \\
  x_{k+1/2} = \ & k \dx,
  &
    y_{k+1/2} = \ & k \dy,
\end{align*}
where $(x_{i+\frac12},y_j)$ and $(x_{i}, y_{j+1/2})$ denote the cells
interfaces and $(x_i, y_j)$ are the cells centers. Set
$N_T = \lfloor T/\dt \rfloor$ and let $t^n = n \, \dt$ ,
$n=0, \ldots, N_T$, be the time mesh. Set $\lambda_x = \dt/\dx$ and
$\lambda_y = \dt / \dy$. 
For the sake of shortness and to improve readability, we use  the notation
$\mathbf{x}_{i,j} = (x_i, y_j)$ in the proofs.

We approximate the initial datum as follows: for $i, \, j \in \interi$ and all $c\in\{1,\dots,N\}$
\begin{displaymath}
  \rho_{i,j}^{c, 0}= \frac{1}{\dx \, \dy} \int_{x_{i-\sfrac12}}^{x_{i+\sfrac12}}
  \int_{y_{j-\sfrac12}}^{y_{j+\sfrac12}} \rho_o^c(x,y) \d{y} \d{x},
\end{displaymath}
and we define a piece-wise constant solution to~\eqref{eq:1} as
\begin{equation}
  \label{eq:rhodelta}
  \rho_{\Delta}^c (t,x,y ) =  \rho_{i,j}^{c,n}
  \quad \mbox{ for } \quad
  \left\{
    \begin{array}{r@{\;}c@{\;}l}
      t & \in &[t^n, t^{n+1}[ \,,
      \\
      x & \in & [x_{i-1/2}, x_{i+1/2}[ \,,
      \\
      y & \in & [y_{j-1/2}, y_{j+1/2}[ \,,
    \end{array}
  \right.
  \quad \mbox{ where } \quad
  \begin{array}{r@{\;}c@{\;}l}
    n & = & 0, \ldots, N_T-1,
    \\
    i & \in & \interi,
    \\
    j & \in & \interi.
  \end{array}
\end{equation}
We compute the convolution products through the following quadrature formula, for $d=1,2$,
\begin{equation} \label{eq:conv}
    \left(\partial_d\eta_c * \left(\sum_{k=1}^{N+1}\alpha_k\rho_\Omega^k\right) \right) (x_i,y_j) =
  \dx \, \dy \sum_{h, \ell \in \interi} \left(\sum_{k=1}^{N+1}\alpha_k\rho^k_{\Omega,h,\ell}\right) \, \partial_d \eta_c(x_{i-h}, y_{j-\ell}),
\end{equation}
where $\partial_1 \eta_c=\partial_x\eta_c$ and
$\partial_2\eta_c=\partial_y\eta_c$, with $\rho^k$ being the density of class $k$.
The approximate solution is then computed through a modified Roe-type scheme with dimensional splitting, as in~\cite{Rossietal}. In other words, the two-dimensional problem is reduced to the sequential resolution of a pair of one-dimensional equations. To this end, we first solve the one-dimensional equation in the $x$-direction
\begin{align}
    \label{eq:scheme1roe}
      \rho_{i,j}^{c,n+1/2} &= \rh{i,j} - \lambda_x \left[
      V_1(\x{i+\sfrac12,j},\rh{i,j}, \rh{i+1,j})
      -
      V_1(\x{i-\sfrac12,j},\rh{i-1,j}, \rh{i,j})
      \right.
    \\ \nonumber
      & \left.
       + F^c (\rh{i,j}, \rh{i+1,j}, J^{c,n}_1(\x{i+\sfrac12,j}))
      - F^c (\rh{i-1,j}, \rh{i,j}, J^{c,n}_1 ( \x{i-\sfrac12,j}))\right]
\end{align}
and then use $\rho_{i,j}^{c,n+1/2}$ to solve in the $y$-direction,
\begin{align}
    \label{eq:scheme2roe}
    \rho_{i,j}^{c,n+1} &=
      \rho_{i,j}^{c,n+1/2} - \lambda_y \left[
      V_2(\x{i, j+\sfrac12},\rho_{i,j}^{c,n+1/2}, \rho_{i,j+1}^{c,n+1/2})
      -
      V_2(\x{i, j-\sfrac12},\rho_{i,j-1}^{c,n+1/2}, \rho_{i,j}^{c,n+1/2})
      \right.
    \\ \nonumber
    &
      \left.
      + F^c (\rho_{i,j}^{c,n+1/2}, \rho_{i,j+1}^{c,n+1/2}, J^{c,n+1/2}_2(\x{i, j+\sfrac12}))
      - F^c (\rho_{i,j-1}^{c,n+1/2}, \rho_{i,j}^{c,n+1/2}, J^{c,n+1/2}_2 ( \x{i,j-\sfrac12}))\right],
\end{align}
where $J^{c,n}_1(t,x,y)$ and $J^{c,n}_2(t,x,y)$ are given by
\begin{align*}
    \boldsymbol{v}_c^{dyn} (\rho^{1,n},...,\rho^{N+1,n}) (x,y) = \left(J_1^{c,n} (x,y), \, J_2^{c,n}
  (x,y)\right).
\end{align*}
and where we define 
\begin{align}
    \label{eq:fx1}
    & V_1 (x,y,u,w) = \vs_1(x,y) \, u + \min\{0, \, \vs_1(x,y)\} (w-u)
    \\
    \label{eq:fx2}
    &F^c (u,w, J_d^c(t,x,y)) =
      J_d^c (t, x,y)u + \min\{0, J_d^c(t,x,y)\}\left(w -u\right)
    \\ \label{eq:gx1}
    & V_2 (x,y,u,w) = \vs_2(x,y) \, u + \min\{0, \, \vs_2(x,y)\} (w-u)
\end{align}
for $n=0,\ldots, N_T-1$, $d=1,2$. 
Remark that the choice of evaluating the numerical flux at $t^n$ for both fractional steps allows to save computational time, because the convolution products~\eqref{eq:conv} are computed only once per time step.

\subsection{Positivity}
\label{sec:pos}

In the case of positive initial datum, we prove that under a suitable
CFL condition the approximate solution to~\eqref{eq:1} constructed via~\eqref{eq:scheme1roe} and~\eqref{eq:scheme2roe} 
remains positive.

\begin{lemma}{\bf (Positivity)} Let
  $\brho_o \in \L\infty (\Omega; \reali^N_+)$. Let~\ref{vs}, \ref{H},
  and~\ref{eta} hold. Assume that
\begin{align}
   \label{eq:CFLroe}
   \lambda_x \leq \
   & \frac{1}{2(\epsilon_{\max} + \norma{\vs_1}_{\L\infty})},
   &
   \lambda_y \leq \
   & \frac{1}{2(\epsilon_{\max} + \norma{\vs_2}_{\L\infty})}.
\end{align}
where $\epsilon_{\max}:=\max_{c}\epsilon_c$. Then, for all $c\in\{1,\dots,N\}$, $t>0$ and $(x,y) \in \reali^2$, the piece-wise constant
approximate solution $\rho_\Delta^c$ constructed
through~\eqref{eq:scheme1roe} and~\eqref{eq:scheme2roe}  
is such that
$\rho_\Delta^c (t,x,y) \geq 0$.
\end{lemma}
\begin{proof}
  Fix $n$ between $0$ and $N_T-1$ and assume that $\rh{i,j} \geq 0$
  for all $i, \, j \in \interi$. We use the notation
\begin{equation}
  \label{eq:notvJ}
  v_{i\pm 1/2} =  \vs_1 (x_{i\pm 1/2,j}).
\end{equation}
  Consider~\eqref{eq:scheme1roe}, with
  the notation~\eqref{eq:fx1} and~\eqref{eq:fx2}, and observe that:
  \begin{align*}
     V_1 (x_{i+1/2,j}, \rh{i,j}, \rh{i+1,j})
      + F^c (\rh{i,j}, \rh{i+1,j}, J^{c,n}_1(\x{i+1/2,j}))
    & \leq 
     v_{i+1/2} \, \rh{i,j}
      + J^{c,n}_1(x_{i+1/2,j}) \rh{i,j}
      \\
     & \leq
      \left(\norma{\vs_1}_{\L\infty} + \epsilon_c \right) \rh{i,j},
    \\
     V_1 (x_{i-1/2,j}, \rh{i-1,j}, \rh{i,j})
      + F^c (\rh{i-1,j}, \rh{i,j}, J^{c,n}_1 (\x{i-1/2,j}))
    &\geq 
    v_{i-1/2} \, \rh{i,j}
      + J^{c,n}_1(x_{i-1/2,j}) \rh{i,j}
      \\
     & \geq
      - \left(\norma{\vs_1}_{\L\infty} + \epsilon_c \right) \rh{i,j},
  \end{align*}
  Therefore, by~\eqref{eq:scheme1roe},
  \begin{displaymath}
    \rho_{i,j}^{c,n+1/2}
    \geq \rh{i,j} - 2\, \lambda_x \left(\norma{\vs_1}_{\L\infty} + \epsilon_c\right) \rh{i,j}
    \geq 0,
  \end{displaymath}
  thanks to the CFL condition~\eqref{eq:CFLroe}. Starting
  from~\eqref{eq:scheme2roe}, an analogous argument shows that
  $\rho_{i,j}^{c,n+1}\geq 0$, concluding the proof.
\end{proof}

\subsection{\texorpdfstring{$\L1$}{L 1} bound}
\label{sec:l1}

The following result on the $\L1$ bound follows from the conservation
property of the Roe scheme.

\begin{lemma}{\bf ($\L1$ bound)}\label{lem:L1roe} Let
  $\brho_o \in (\L1\cap\L\infty) (\Omega; \reali^N_+)$. Let~\ref{vs}, \ref{H},
  \ref{eta} and~\eqref{eq:CFLroe} hold. Then, for all $c\in\{1,...,N\}$, $t>0$ and
  $(x,y) \in \reali^2$, $\rho^c_\Delta$ constructed
  through ~\eqref{eq:scheme1roe} and~\eqref{eq:scheme2roe} 
  satisfies
  \begin{equation}
    \label{eq:l1}
    \norma{\rho^c_\Delta(t,\cdot,\cdot)}_{\L1(\Omega)} = \norma{\rho^c_o}_{\L1(\Omega)}.
  \end{equation}
\end{lemma}

\subsection{\texorpdfstring{$\L\infty$}{L infinity} bound}
\label{sec:linf}

\begin{lemma}{\bf ($\L\infty$~bound)}\label{lem:Linfroe} Let
  $\brho_o \in \L\infty (\Omega; \reali^N_+)$. Let~\ref{vs}, \ref{H},
  \ref{eta} and~\eqref{eq:CFLroe} hold. Then, for all $c\in\{1,...,N\}$, $t>0$ and
  $(x,y) \in \reali^2$, $\rho^c_\Delta$ constructed
  through~\eqref{eq:scheme1roe} and~\eqref{eq:scheme2roe}  
  satisfies
  \begin{equation}
    \label{eq:linf}
     \norma{\rho^c_\Delta (t,\cdot,\cdot)}_{\L\infty (\Omega)}
    \leq  \norma{\rho^c_o}_{\L\infty(\Omega)} \, e^{\mathcal{C}^c_\infty \, t},
  \end{equation}
  where
  \begin{equation}
    \label{eq:Cinf}
     \mathcal{C}^c_\infty =
      \norma{\partial_x \vs_1}_{\L\infty}
      + \norma{\partial_y \vs_2}_{\L\infty}
      + 2 \, \epsilon_c \,\left(2 \,
      \norma{\nabla^2 \eta_c}_{\L\infty} 
      + L_H \norma{\nabla\Tilde{\eta}_c}_{\L\infty}
      \right) \norma{r_{\Omega,o}}_{\L1}.
  \end{equation}
\end{lemma}
\begin{proof}
  Fix $c\in\{1,\dots,N\}$. Setting $r_{\Omega,o}=\sum_{k=1}^{N+1}\alpha_k\,\rho^k_{\Omega,o}$ and exploiting the notation introduced in~\eqref{eq:notvJ}, we
  observe that to get an upper estimate for \eqref{eq:scheme1roe} we can assume $v_{i+1/2} < 0$, $v_{i-1/2}\geq 0$, $ J_1^{c,n}(x_{i+1/2})< 0$ and $ J_1^{c,n}(x_{i-1/2})\geq 0$. In this case
  \begin{align*}
    \rho_{i,j}^{c,n+1/2}
    \leq \
       \rh{i,j} &- \lambda_x \left(
        \rh{i+1,j} \, v_{i+1/2} + J_1^{c,n}(x_{i+1/2,j}) \, \rh{i+1,j}
        \right) \\
        &+ \lambda_x \left(
        \rh{i-1,j} \, v_{i-1/2} + J_1^{c,n}(x_{i-1/2,j}) \, \rh{i-1,j}
        \right),
  \end{align*}
  where we used the positivity of  $\rh{i,j}$
  and discarded all the terms giving a negative contribution.  Moreover, since $v_{i+1/2} < 0$ and $v_{i-1/2}\geq 0$,
  \begin{align*}
    \lambda_x \left(
    - \rh{i+1,j} \, v_{i+1/2} + \rh{i-1,j} \, v_{i-1/2}
    \right)
    \leq \
    & \lambda_x \, \norma{\rho^{c,n}}_{\L\infty} \left(
      -v_{i+1/2} + v_{i-1/2}
      \right)
    \\
    = \
    &
      \lambda_x \, \norma{\rho^{c,n}}_{\L\infty} \, (- \dx) \, \partial_x \vs_1(\hat x_i)
  \end{align*}
  with $\hat x_i \in \, ]x_{i-1/2}, x_{i+1/2}[$. In a similar way,
  since $J_1^{c,n}(x_{i+1/2})< 0$ and $J_1^{c,n}(x_{i-1/2})\geq 0$, we
  get
  \begin{align*}
  \begin{split}
    &\lambda_x \left(
    -  J_1^{c,n}(x_{i+1/2}) \, \rh{i+1,j} +  J_1^{c,n}(x_{i-1/2}) \, \rh{i-1,j}
    \right)
    \\
    &\quad
    \leq \
     \lambda_x \, \norma{\rho^{c,n}}_{\L\infty} \left(
      -  J_1^{c,n}(x_{i+1/2,j}) +  J_1^{c,n}(x_{i-1/2,j})
      \right)
    \\
    &\quad
    \leq \
    \lambda_x \, \norma{\rho^{c,n}}_{\L\infty} \, 
    \epsilon_c \, \dx \,\left(2 \, 
      \norma{\nabla^2 \eta_c}_{\L\infty} 
      + L_H \norma{\nabla\tilde{\eta}_c}_{\L\infty}
      \right)\norma{r_\Omega^{n}}_{\L1},
    \end{split}
  \end{align*}
  thanks to~\eqref{eq:Jx}.  Therefore,
  \begin{align*}
    \rho_{i,j}^{c,n+1/2} \leq \
    \norma{\rho^{c,n}}_{\L\infty} \left[
    1 + \dt \left(
    \norma{\partial_x \vs_1}_{\L\infty} + \epsilon_c \,\left(2 \,
      \norma{\nabla^2 \eta_c}_{\L\infty} 
      + L_H \norma{\nabla\Tilde{\eta}_c}_{\L\infty}
      \right) \norma{r_{\Omega,o}}_{\L1}
    \right)
    \right].
  \end{align*} 
  In a similar way we get
  \begin{align*}
    \rho_{i,j}^{c,n+1} \leq \
    \norma{\rho^{c,n+1/2}}_{\L\infty} \!\! \left[
    1 + \dt \left(
    \norma{\partial_y \vs_2}_{\L\infty} + \epsilon_c \left(2 
      \norma{\nabla^2 \eta_c}_{\L\infty} \!\! 
      + L_H \norma{\nabla\Tilde{\eta}_c}_{\L\infty}
      \right) \norma{r_{\Omega,o}}_{\L1}
    \right)
    \right].
  \end{align*}
  An iterative argument completes the proof.
\end{proof}

\subsection{\texorpdfstring{$\BV$}{BV} bound}
\label{sec:bv}

\begin{proposition} {\bf ($\BV$ estimate in space)}\label{prop:bvroe}
  Let $\brho_o \in (\L1\cap \L\infty \cap \BV) (\Omega;
  \reali^N_+)$. Let~\ref{vs}, \ref{H}, ~\ref{eta} and~\eqref{eq:CFLroe} hold. Then, for all $c\in\{1,\dots,N\}$, $t>0$, $\rho^c_\Delta$ 
  constructed through~\eqref{eq:scheme1roe} and~\eqref{eq:scheme2roe} satisfies the following
  estimate: for all $n=0, \ldots, N_T$,
  \begin{equation}
    \label{eq:bvspaceroe}
    \sum_{i,j \in \interi} \left(
      \dy \, \modulo{\rh{i+1,j} - \rh{i,j}}
      + \dx \, \modulo{\rh{i,j+1} - \rh{i,j}}\right)
    \leq \mathcal{C}^c_{x}(t^n),
  \end{equation}
  where
   \begin{equation}
    \label{eq:Cx}
    \mathcal{C}^c_{x}(t) = e^{2 \,t \, \mathcal{K}^c_1}  \sum_{i,j \in \interi} \left(
      \dx \, \modulo{\rho^{c,0}_{i,j+1} - \rho^{c,0}_{i,j}}
      +  \dy \, \modulo{\rho^{c,0}_{i+1,j} -\rho^{c,0}_{i,j}}
    \right)
    +  \frac{2 \, \mathcal{K}^c_2}{\mathcal{K}^c_1}\left(e^{2 \, t \, \mathcal{K}^c_1} -1
    \right),
  \end{equation}
  with
  \begin{align}
    \label{eq:K1defroe}
    \mathcal{K}^c_1 = \
    &
    6  \left( \norma{\nabla \boldsymbol{\vs}}_{\L\infty} +  \epsilon_c \, \left(2\,
    \norma{\nabla^2 \eta_c}_{\L\infty}+ L_H\,\norma{\nabla\Tilde{\eta}_c}_{\L\infty}\right) \norma{r_{\Omega,o}}_{\L1}\right),
    \\
    \label{eq:K2defroe}
    \mathcal{K}^c_2 = \
    &  \left(
      4 \, \epsilon_c \left(
        c^c_1 \,\norma{r_{\Omega,o}}_{\L1} + c^c_2 \,\norma{r_{\Omega,o}}_{\L1}^2
      \right)
      + 3 \, \norma{\nabla^2 \boldsymbol{\vs}}_{\L\infty}
      \right) \norma{\rho^c_o}_{\L1(\Omega)},
  \end{align}
  and $c^c_1, \, c^c_2 $ are defined in~\eqref{eq:c12}.
\end{proposition}

\begin{proof}
  We follow the proof in~\cite{Rossietal} adapting it to the system. First consider for $c\in\{1,...,N\}$ the
  term
  \begin{displaymath}
    \sum_{i,j \in \interi} \dy\, \modulo{\rho^{c,n+1/2}_{i+1,j} - \rho^{c,n+1/2}_{i,j}}.
  \end{displaymath}
  In particular, fixing $i, j \in \interi$ and omitting the
  dependencies on $y_j$ for the sake of simplicity and setting $r_{\Omega,o}=\sum_{k=1}^{N+1}\alpha_k\,\rho^k_{\Omega,o} $,
  by~\eqref{eq:scheme1roe} we get
  \begin{align*}
    \rho_{i+1}^{c,n+1/2} - \rho_{i}^{c,n+1/2}
    = \
    & \rh{i+1} - \rh{i}
      - \lambda_x \left[ V_1 (\x{i+3/2}, \rh{i+1}, \rh{i+2})
      +F^c (\rh{i+1}, \rh{i+2}, J^{c,n}_1 (x_{i+3/2}))
      \right.
    \\
    & \qquad\qquad\qquad\quad
      -V_1 (\x{i+1/2}, \rh{i}, \rh{i+1})
      -F^c (\rh{i}, \rh{i+1}, J^{c,n}_1 (x_{i+1/2}))
    \\
    & \qquad\qquad\qquad\quad
      - V_1 (\x{i+1/2}, \rh{i}, \rh{i+1})
      - F^c (\rh{i}, \rh{i+1}, J^{c,n}_1 (x_{i+1/2}))
    \\
    & \qquad\qquad\qquad\quad
      \left.
      +  V_1 (\x{i-1/2}, \rh{i-1}, \rh{i})
      +F^c (\rh{i-1}, \rh{i}, J^{c,n}_1 (x_{i-1/2}))
      \right]
    \\
    & \pm \lambda_x \left[
      V_1 (\x{i+3/2}, \rh{i}, \rh{i+1})
      +F^c (\rh{i}, \rh{i+1}, J^{c,n}_1 (x_{i+3/2}))
      \right.
    \\
    &  \qquad\left.
      -V_1 (\x{i+1/2}, \rh{i-1}, \rh{i})
      -F^c (\rh{i-1}, \rh{i}, J^{c,n}_1 (x_{i+1/2}))\right]
    \\
    = \
    & \mathcal{A}_{i,j}^{c,n} - \lambda_x \, \mathcal{B}_{i,j}^{c,n},
  \end{align*}
  where we set
\begin{align*}
    \mathcal{A}_{i,j}^{c,n} = \
    & \rh{i+1} - \rh{i} - \lambda_x \left[
      V_1 (\x{i+3/2}, \rh{i+1}, \rh{i+2})
      +F^c (\rh{i+1}, \rh{i+2}, J^{c,n}_1 (x_{i+3/2})) \right.
    \\
    & \qquad\qquad\qquad\quad
      -V_1 (\x{i+1/2}, \rh{i}, \rh{i+1})
      -F^c (\rh{i}, \rh{i+1}, J^{c,n}_1 (x_{i+1/2}))
    \\
    & \qquad\qquad\qquad\quad
      +V_1 (\x{i+1/2}, \rh{i-1}, \rh{i})
      +F^c (\rh{i-1}, \rh{i}, J^{c,n}_1 (x_{i+1/2}))
    \\
    & \qquad\qquad\qquad\quad
      \left.
      -V_1 (\x{i+3/2}, \rh{i}, \rh{i+1})
      -F^c (\rh{i}, \rh{i+1}, J^{c,n}_1 (x_{i+3/2}))\right],
    \\
    \mathcal{B}_{i,j}^{c,n} = \
    &  V_1 (\x{i+3/2}, \rh{i}, \rh{i+1})
      +F^c (\rh{i}, \rh{i+1}, J^{c,n}_1 (x_{i+3/2}))
    \\
    &  -V_1 (\x{i+1/2}, \rh{i-1}, \rh{i})
      -F^c (\rh{i-1}, \rh{i}, J^{c,n}_1 (x_{i+1/2}))
    \\
    &+  V_1 (\x{i-1/2}, \rh{i-1}, \rh{i})
      +F^c (\rh{i-1}, \rh{i}, J^{c,n}_1 (x_{i-1/2}))
    \\
    &    -V_1 (\x{i+1/2}, \rh{i}, \rh{i+1})
      -F^c (\rh{i}, \rh{i+1}, J^{c,n}_1 (x_{i+1/2})).
  \end{align*}
  For the sake of shortness, introduce the following notation
  \begin{equation}
    \label{eq:10}
    H^{c,n}_{h,\ell}(u,w) = V_1(x_{h,\ell},u,w)+F^c(u,w,J^{c,n}_1(x_{h,\ell})),
  \end{equation}
  so that, dropping the $j$ dependencies, $\mathcal{A}^{k,n}_{i,j}$ reads
  \begin{align*}
    \mathcal{A}_{i,j}^{c,n}= & \ \rh{i+1} - \rh{i} - \lambda_x
     \left[
      H^{c,n}_{i+3/2} (\rh{i+1},\rh{i+2})
      -H^{c,n}_{i+1/2} (\rh{i},\rh{i+1})\right.
    \\
    & \left.
      +H^{c,n}_{i+1/2} (\rh{i-1},\rh{i})
      -H^{c,n}_{i+3/2} (\rh{i},\rh{i+1})\right]
    \\
   = \
    & \rh{i+1} - \rh{i}
      - \lambda_x \, \frac{
      H^{c,n}_{i+3/2} ( \rh{i+1},\rh{i+2})
      - H^{c,n}_{i+3/2} ( \rh{i+1},\rh{i+1}) }{\rh{i+2} - \rh{i+1}}
      \left(\rh{i+2}-\rh{i+1}\right)
    \\
    &
      -\lambda_x \, \frac{H^{c,n}_{i+3/2}( \rh{i+1},\rh{i+1})
      - H^{c,n}_{i+3/2} (\rh{i},\rh{i+1})}{\rh{i+1}-\rh{i}}
      \left(\rh{i+1}-\rh{i}\right)
    \\
    & + \lambda_x \, \frac{ H^{c,n}_{i+1/2}( \rh{i}, \rh{i+1})
      -  H^{c,n}_{i+1/2} (\rh{i}, \rh{i})}{\rh{i+1} - \rh{i}}
      \left(\rh{i+1}-\rh{i}\right)
    \\
    & + \lambda_x \, \frac{ H^{c,n}_{i+1/2}(\rh{i}, \rh{i})
      -  H^{c,n}_{i+1/2}( \rh{i-1}, \rh{i}) }{\rh{i} - \rh{i-1}}
      \left(\rh{i} - \rh{i-1} \right)
    \\
    = \
    & \delta_{i+1}^{c,n}  \left(\rh{i+2}-\rh{i+1}\right)
      + \theta_i^{c,n}  \left(\rh{i} - \rh{i-1} \right)
      + (1 - \delta_i^{c,n} - \theta_{i+1}^{c,n}) \left(\rh{i+1}-\rh{i}\right),
  \end{align*}
  where
  \begin{align}
    \label{eq:deltainroe}
    \delta_i^{c,n} = \ &
                     \begin{cases}
                       - \lambda_x \, \dfrac{H^{c,n}_{i+1/2} ( \rh{i},
                         \rh{i+1})- H^{c,n}_{i+1/2} (\rh{i},
                         \rh{i})}{\rh{i+1}-\rh{i}} & \mbox{if } \rh{i}
                       \neq \rh{i+1},
                       \\
                       0 & \mbox{if } \rh{i} = \rh{i+1} ,
                     \end{cases}
    \\
    \label{eq:thetainroe}
    \theta_i^{c,n} = \ &
                     \begin{cases}
                       \lambda_x \, \dfrac{H^{c,n}_{i+1/2} (\rh{i},
                         \rh{i}) - H^{c,n}_{i+1/2}(\rh{i-1},
                         \rh{i})}{\rh{i}-\rh{i-1}} & \mbox{if } \rh{i}
                       \neq \rh{i-1},
                       \\
                       0 & \mbox{if } \rh{i} = \rh{i-1}.
                     \end{cases}
  \end{align}
  Exploiting~\eqref{eq:10}, observe that, whenever
  $\rh{i}\neq \rh{i+1}$,
  \begin{align*}
    \delta_i^{c,n} = \
    & -\frac{\lambda_x}{\rh{i+1}-\rh{i}}
      \left[
      V_1(x_{i+1/2},\rh{i}, \rh{i+1})
      + F^c (\rh{i}, \rh{i+1}, J^{c,n}_1 (x_{i+1/2}))\right.
    \\
    & \qquad\qquad\qquad\left.
      -V_1(x_{i+1/2},\rh{i}, \rh{i})
      - F^c (\rh{i}, \rh{i}, J^{c,n}_1 (x_{i+1/2}))
      \right]
    \\
    = \
    & \! -\frac{\lambda_x}{\rh{i+1}-\rh{i}}\!\!
      \left[ \min\left\{0, \vs_1(x_{i+1/2})\right\} (\rh{i+1} - \rh{i})
      +  \min\left\{0, J^{c,n}_1(x_{i+1/2})\right\} \!\!\left(\rh{i+1} -\rh{i}\right)\!
      \right]
    \\
    = \
    & -\lambda_x \left(  \min\left\{0, \vs_1(x_{i+1/2})\right\}
      + \min\left\{0, J^{c,n}_1(x_{i+1/2})\right\} \right).
  \end{align*}
  \\
  Since $|J^{c,n}_1|\le\varepsilon_c$ by \eqref{eq:Jinf}, we get
\[
  \delta^{c,n}_i\in\Big[0,\ \lambda_x\big(\|v^{stat}_1\|_{\mathbf{L}^\infty}+\varepsilon_{\max}\big)\Big].
\]
In a similar way, whenever $\rho^{c,n}_i\neq\rho^{c,n}_{i-1}$, \eqref{eq:thetainroe} gives
\[
  \vartheta^{c,n}_i=\lambda_x\Big(\max\big\{0,v^{stat}_1(x_{i+1/2})\big\}
    +\max\big\{0,J^{c,n}_1(x_{i+1/2})\big\}\Big)
  \in\Big[0,\ \lambda_x\big(\|v^{stat}_1\|_{\mathbf{L}^\infty}+\varepsilon_{\max}\big)\Big].
\]
By the CFL condition \eqref{eq:CFLroe}, we therefore have
$\delta^{c,n}_i,\ \vartheta^{c,n}_i\in\left[0,\tfrac12\right]$ for all $i\in\mathbb{Z}$,
and in particular $1-\delta^{c,n}_i-\vartheta^{c,n}_{i+1}\ge 0$.
  Thus,
  \begin{equation}
    \label{eq:AijOKroe}
    \sum_{i,j\in \interi} \modulo{\mathcal{A}_{i,j}^{c,n}}
    \leq
    \sum_{i,j \in \interi} \modulo{\rh{i+1,j} - \rh{i,j}}.
  \end{equation}

We pass now to $\mathcal{B}_{i,j}^{c,n}$. Consider separately the terms
  involving $V_1$ and those involving $F^c$. Observe that the maps
  \begin{align*}
    x \mapsto &\min\left\{0, \vs_1 (x)\right\},
    &
    x \mapsto &\min\left\{0, J_1^{c,n} (x)\right\}
  \end{align*}
  are Lipschitz continuous, with constant respectively
  $\norma{\partial_x \vs_1}_{\L\infty}$ and
  \\
  $\epsilon_c\, \left(2 \, \norma{\nabla^2\eta_c}_{\L\infty}+ L_H \norma{\nabla \Tilde{\eta}_c}_{\L\infty}\right)
  \norma{r_{\Omega,o}}_{\L1}$.  Exploiting~\eqref{eq:fx1} we get:
  \begin{align}
    \nonumber
    & V_1 (\x{i+3/2}, \rh{i}, \rh{i+1})
      -V_1 (\x{i+1/2}, \rh{i-1}, \rh{i})
      +  V_1 (\x{i-1/2}, \rh{i-1}, \rh{i})
      -V_1 (\x{i+1/2}, \rh{i}, \rh{i+1})
    \\  \nonumber
    = \
    & \vs_1 (x_{i+3/2}) \rh{i} + \min\left\{0, \vs_1 (x_{i+3/2})\right\} (\rh{i+1} - \rh{i})
    \\  \nonumber
    &  - \vs_1 (x_{i+1/2}) \rh{i} - \min\left\{0, \vs_1 (x_{i+1/2})\right\} (\rh{i+1} - \rh{i})
    \\  \nonumber
    &   + \vs_1 (x_{i-1/2}) \rh{i-1} + \min\left\{0, \vs_1 (x_{i-1/2})\right\} (\rh{i} - \rh{i-1})
    \\ \nonumber
    & - \vs_1 (x_{i+1/2}) \rh{i-1} - \min\left\{0, \vs_1 (x_{i+1/2})\right\} (\rh{i} - \rh{i-1})
    \\ \nonumber
    & \pm \left(\vs_1 (x_{i-1/2}) - \vs_1 (x_{i+1/2})\right) \rh{i}
    \\ \nonumber
    = \
    &  \left(\vs_1 (x_{i+3/2})  -2 \, \vs_1 (x_{i+1/2}) + \vs_1 (x_{i-1/2})\right) \rh{i}
    \\ \nonumber
    & +  \left(\vs_1 (x_{i-1/2}) - \vs_1 (x_{i+1/2})\right)(\rh{i-1} - \rh{i})
    \\ \nonumber
    &+ \left(
      \min\left\{0, \vs_1 (x_{i+3/2})\right\} - \min\left\{0, \vs_1 (x_{i+1/2})\right\}
      \right)(\rh{i+1} - \rh{i})
    \\ \nonumber
    & +\left(
      \min\left\{0, \vs_1 (x_{i-1/2})\right\} -   \min\left\{0, \vs_1 (x_{i+1/2})\right\}
      \right)(\rh{i} - \rh{i-1})
    \\ \label{eq:B-V}
    \leq \
    & 2 \, (\dx)^2 \, \norma{\partial_{xx}^2 \vs_1}_{\L\infty} \modulo{\rh{i}}
      + \dx \norma{\partial_x \vs_1}_{\L\infty} \left(
      \modulo{\rh{i+1} - \rh{i}} + 2 \, \modulo{\rh{i} - \rh{i-1}}
      \right),
  \end{align}
  since
  \begin{align*}
    \vs_1 (x_{i+3/2})  -2 \, \vs_1 (x_{i+1/2}) + \vs_1 (x_{i-1/2}) = \
    &
      \dx \, \partial_x \vs_1 (\xi_{i+1}) - \dx \, \partial_x \vs_1 (\xi_{i})
    \\
    = \
    &
      \dx \, (\xi_{i+1} - \xi_i) \, \partial_{xx}^2\vs_1 (\zeta_{i+1/2}),
  \end{align*}
  with $\xi_i \in\, ]x_{i-1/2}, x_{i+1/2}[ $ and
  $\zeta_{i+1/2} \in \,]\xi_i, \xi_{i+1}[$.
  Similarly, exploiting~\eqref{eq:fx2} we obtain
  \begin{align}
    \nonumber
    & F^c (\rh{i}, \rh{i+1}, J^{c,n}_1 (x_{i+3/2}))
      -F^c (\rh{i-1}, \rh{i}, J^{c,n}_1 (x_{i+1/2}))
    \\
    \nonumber
    & +F^c (\rh{i-1}, \rh{i}, J^{c,n}_1 (x_{i-1/2}))
      -F^c (\rh{i}, \rh{i+1}, J^{c,n}_1 (x_{i+1/2}))
    \\ \nonumber
    = \
    & J^{c,n}_1 (x_{i+3/2}) \rh{i} + \min\left\{0, J^{c,n}_1 (x_{i+3/2})\right\} (\rh{i+1} - \rh{i})
    \\  \nonumber
    &  - J^{c,n}_1 (x_{i+1/2}) \rh{i} - \min\left\{0, J^{c,n}_1 (x_{i+1/2})\right\} (\rh{i+1} - \rh{i})
    \\  \nonumber
    &   + J^{c,n}_1 (x_{i-1/2}) \rh{i-1} + \min\left\{0, J^{c,n}_1 (x_{i-1/2})\right\} (\rh{i} - \rh{i-1})
    \\ \nonumber
    & - J^{c,n}_1 (x_{i+1/2}) \rh{i-1} - \min\left\{0, J^{c,n}_1 (x_{i+1/2})\right\} (\rh{i} - \rh{i-1})
    \\ \nonumber
    & \pm \left(J^{c,n}_1 (x_{i-1/2}) - J^{c,n}_1 (x_{i+1/2})\right) \rh{i}
    \\ \nonumber
    = \
    &  \left(J^{c,n}_1 (x_{i+3/2})  -2 \, J^{c,n}_1 (x_{i+1/2}) + J^{c,n}_1 (x_{i-1/2})\right) \rh{i}
    \\ \nonumber
    & +  \left(J^{c,n}_1 (x_{i-1/2}) - J^{c,n}_1 (x_{i+1/2})\right)(\rh{i-1} - \rh{i})
    \\ \nonumber
    &+ \left(
      \min\left\{0, J^{c,n}_1 (x_{i+3/2})\right\} - \min\left\{0, J^{c,n}_1 (x_{i+1/2})\right\}
      \right)(\rh{i+1} - \rh{i})
    \\ \nonumber
    & +\left(
      \min\left\{0, J^{c,n}_1 (x_{i-1/2})\right\} -   \min\left\{0, J^{c,n}_1 (x_{i+1/2})\right\}
      \right)(\rh{i} - \rh{i-1})
    \\ \label{eq:B-F}
    \leq \
    & 2 \, \epsilon_c \, (\dx)^2 \, C_c \, \modulo{\rh{i}}
      \\
      \nonumber
     & +  \epsilon_c \,  \dx \left(2\,\norma{\nabla^2 \eta_c}_{\L\infty} + L_H\norma{\nabla\Tilde{\eta}_c}_{L\infty}\right) \norma{r_{\Omega,o}}_{\L1} \left(
      \modulo{\rh{i+1} - \rh{i}} + 2 \, \modulo{\rh{i} - \rh{i-1}}
      \right),
  \end{align}
  \\
  where we used \eqref{eq:Jx} and~\eqref{eq:Jtripla}, with the
  notation~\eqref{eq:c12}.  Collecting together~\eqref{eq:B-V}
  and~\eqref{eq:B-F} we therefore obtain
  \begin{align*}
    \modulo{\mathcal{B}_{i,j}^{c,n}}\leq \
    &  2 \, (\dx)^2 \left(
      \norma{\partial_{xx}^2 \vs_1}_{\L\infty}
      + \epsilon_c\, C_c
      \right)\modulo{\rh{i}} + \dx  \left(
      \norma{\partial_x \vs_1}_{\L\infty}\right.
    \\
    &   \left.
      +  \epsilon_c  \left(2\,\norma{\nabla^2 \eta_c}_{\L\infty} + L_H\norma{\nabla\Tilde{\eta}_c}_{\L\infty}\right) \norma{r_{\Omega,o}}_{\L1}
      \right)
      \left(
      \modulo{\rh{i+1} - \rh{i}} + 2 \, \modulo{\rh{i} - \rh{i-1}}
      \right),
  \end{align*}
  so that
    \begin{align}
    \sum_{i,j \in \interi} & \lambda_x   \, \modulo{\mathcal{B}_{i,j}^{c,n}}
    \leq \
       2 \, \dt  \left(
      \norma{\partial_{xx}^2 \vs_1}_{\L\infty}
      + \epsilon_c\, C_c
      \right) \dx \sum_{i,j \in \interi} \modulo{\rh{i}}
      \label{eq:BijOKroe}
      \\
      \nonumber
    & + 3 \, \dt  \left(
      \norma{\partial_x \vs_1}_{\L\infty}
      +  \epsilon_c  \left(2\,\norma{\nabla^2 \eta_c}_{\L\infty} + L_H\norma{\nabla\Tilde{\eta}_c}_{\L\infty}\right) \norma{r_{\Omega,o}}_{\L1}
      \right) \sum_{i,j \in \interi} \modulo{\rh{i+1} -\rh{i}}.
  \end{align}
  Therefore, by~\eqref{eq:AijOKroe} and~\eqref{eq:BijOKroe}, using
  also Lemma~\ref{lem:L1roe}
  \begin{align}
    \nonumber
    & \sum_{i,j \in \interi} \dy \, \modulo{\rho_{i+1,j}^{c,n+1/2} - \rho_{i,j}^{c,n+1/2}}
    \\ \nonumber
    \leq \
    & \sum_{i,j \in \interi} \dy \left(\modulo{\mathcal{A}_{i,j}^{c,n}}
      + \lambda_x\, \modulo{\mathcal{B}_{i,j}^{c,n}}\right)
    \\
    \label{eq:BVdyroe}
    \leq \
    & \left[ 1+ 3 \, \dt  \left(
      \norma{\partial_x \vs_1}_{\L\infty}
      +  \epsilon_c  \left(2\,\norma{\nabla^2 \eta_c}_{\L\infty} + L_H\norma{\nabla\Tilde{\eta}_c}_{\L\infty}\right) \norma{r_{\Omega,o}}_{\L1}
      \right) \right]
    \\ \nonumber
    &
      \quad\quad \cdot \sum_{i,j \in \interi} \dy \, \modulo{\rh{i+1} -\rh{i}} + 2 \, \dt  \left[
      \norma{\partial_{xx}^2 \vs_1}_{\L\infty}
      + \epsilon_c \, C_c
      \right] \norma{\rho^c_o}_{\L1}.
  \end{align}
  Now pass to the term
  \begin{displaymath}
    \sum_{i,j \in \interi} \dx \, \modulo{\rho_{i,j+1}^{c,n+1/2} - \rho_{i,j}^{c,n+1/2}}.
  \end{displaymath}
  Fix $i,j\in\interi$ and exploit~\eqref{eq:scheme1roe} again
  to get
  \begin{align*}
    &\rho^{c,n+1/2}_{i,j+1} - \rho^{c,n+1/2}_{i,j}
    \\
    = \
    & \rh{i,j+1} - \rh{i,j}
      - \lambda_x \left[
      V_1 (\x{i+1/2,j+1}, \rh{i,j+1}, \rh{i+1,j+1})
      +F^c (\rh{i,j+1}, \rh{i+1,j+1}, J^{c,n}_1 (x_{i+1/2,j+1}))
      \right.
    \\
    & \qquad\qquad\qquad\quad
      -V_1 (\x{i-1/2,j+1}, \rh{i-1,j+1}, \rh{i,j+1})
      -F^c (\rh{i-1,j+1}, \rh{i,j+1}, J^{c,n}_1 (\x{i-1/2,j+1}))
    \\
    & \qquad\qquad\qquad\quad
      - V_1 (\x{i+1/2,j}, \rh{i,j}, \rh{i+1,j})
      - F^c (\rh{i,j}, \rh{i+1,j}, J^{c,n}_1 (x_{i+1/2,j}))
    \\
    & \qquad\qquad\qquad\quad
      \left.
      +  V_1 (\x{i-1/2,j}, \rh{i-1,j}, \rh{i,j})
      +F^c (\rh{i-1,j}, \rh{i,j}, J^{c,n}_1 (x_{i-1/2,j}))
      \right]
    \\
    & \pm \lambda_x \left[
      V_1 (\x{i+1/2,j+1}, \rh{i,j}, \rh{i+1,j})
      +F^c (\rh{i,j}, \rh{i+1,j}, J^{c,n}_1 (x_{i+1/2,j+1}))
      \right.
    \\
    &  \qquad\left.
      -V_1 (\x{i-1/2,j+1}, \rh{i-1,j}, \rh{i,j})
      -F^c (\rh{i-1,j}, \rh{i,j}, J^{c,n}_1 (x_{i-1/2,j+1}))\right]
    \\
    = \
    & \mathcal{D}_{i,j}^{c,n} + \lambda_x \, \mathcal{E}_{i,j}^{c,n},
  \end{align*}
  where we set
\begin{align*}
     \mathcal{D}_{i,j}^{c,n}
     = \ & \rh{i,j+1} - \rh{i,j}\\ &-\lambda_x \left[
      V_1 (\x{i+1/2,j+1}, \rh{i,j+1}, \rh{i+1,j+1})
      + F^c (\rh{i,j+1}, \rh{i+1,j+1}, J_1^{c,n} (\x{i+1/2,j+1}))
      \right.
    \\ &  
         - V_1 (\x{i+1/2,j+1}, \rh{i,j}, \rh{i+1,j})
         - F^c ( \rh{i,j}, \rh{i+1,j}, J_1^{c,n} (\x{i+1/2,j+1}))
    \\
    &  
      +  V_1 (\x{i-1/2,j+1}, \rh{i-1,j}, \rh{i,j})
      + F^c( \rh{i-1,j}, \rh{i,j}, J_1^{c,n} (\x{i-1/2,j+1}))
    \\
    & \left. 
      - V_1 (\x{i-1/2,j+1}, \rh{i-1,j+1}, \rh{i,j+1})
      - F^c(\rh{i-1,j+1}, \rh{i,j+1}, J_1^{c,n}(\x{i-1/2,j+1}))
      \right]
  \end{align*}
  and
  \begin{align*}
    \mathcal{E}_{i,j}^{c,n} = \
    & V_1 (\x{i+1/2,j}, \rh{i,j}, \rh{i+1,j})
      + F^c (\rh{i,j}, \rh{i+1,j}, J_1^{c,n} (\x{i+1/2,j}))
    \\
    & - V_1 (\x{i+1/2,j+1}, \rh{i,j}, \rh{i+1,j})
      - F^c (\rh{i,j}, \rh{i+1,j}, J_1^{c,n} (\x{i+1/2,j+1}))
    \\
    &
      + V_1 (\x{i-1/2,j+1}, \rh{i-1,j}, \rh{i,j})
      + F^c (\rh{i-1,j}, \rh{i,j}, J_1^{c,n} (\x{i-1/2,j+1}))
    \\
    & - V_1 (\x{i-1/2,j}, \rh{i-1,j}, \rh{i,j})
      - F^c (\rh{i-1,j}, \rh{i,j}, J_1^{c,n} (\x{i-1/2,j})).
  \end{align*}
  Similarly as before, rearrange $\mathcal{D}_{i,j}^{c,n}$, exploiting the
  notation~\eqref{eq:10}:
  \begin{align*}
    \mathcal{D}_{i,j}^{c,n}
    = \
    & \rh{i,j+1} - \rh{i,j}
      - \lambda_x\left[
      H^{c,n}_{i+1/2,j+1} (\rh{i,j+1}, \rh{i+1,j+1})
      - H^{c,n}_{i+1/2,j+1} (\rh{i,j}, \rh{i+1,j})
      \right.
    \\
    & \left. \qquad\qquad\qquad\qquad + H^{c,n}_{i-1/2,j+1} (\rh{i-1,j}, \rh{i,j})
      - H^{c,n}_{i-1/2,j+1} (\rh{i-1,j+1}, \rh{i,j+1})
      \right]
    \\
    & \pm \, \lambda_x H^{c,n}_{i+1/2, j+1} (\rh{i,j}, \rh{i+1,j+1})
      \pm \, \lambda_x H^{c,n}_{i-1/2,j+1} (\rh{i-1,j}, \rh{i,j+1})
    \\
    = \
    & \rh{i,j+1} - \rh{i,j}
    \\
    & - \lambda_x \frac{H^{c,n}_{i+1/2,j+1} (\rh{i,j+1}, \rh{i+1,j+1})
      - H^{c,n}_{i+1/2,j+1} (\rh{i,j}, \rh{i+1,j+1}) }
      {\rh{i, j+1}-\rh{i,j}}(\rh{i, j+1}-\rh{i,j})
    \\
    & - \lambda_x \frac{H^{c,n}_{i+1/2,j+1} (\rh{i,j}, \rh{i+1,j+1})
      - H^{c,n}_{i+1/2,j+1} (\rh{i,j}, \rh{i+1,j}) }
      {\rh{i+1, j+1}-\rh{i+1,j}}(\rh{i+1, j+1}-\rh{i+1,j})
    \\
    & + \lambda_x \frac{H^{c,n}_{i-1/2,j+1} (\rh{i-1,j}, \rh{i,j+1})
      - H^{c,n}_{i-1/2,j+1} (\rh{i-1,j}, \rh{i,j}) }
      {\rh{i, j+1}-\rh{i,j}}(\rh{i, j+1}-\rh{i,j})
    \\
    & + \lambda_x \frac{H^{c,n}_{i-1/2,j+1} ( \rh{i-1,j+1}, \rh{i,j+1})
      - H^{c,n}_{i-1/2,j+1} ( \rh{i-1,j}, \rh{i,j+1}) }
      {\rh{i-1, j+1}-\rh{i-1,j}}(\rh{i-1, j+1}-\rh{i-1,j})
    \\
    = \
    & (1 - \kappa_{i,j}^{c,n} - \nu_{i,j}^{c,n}) (\rh{i, j+1}-\rh{i,j})
      + \nu_{i+1,j}^{c,n} (\rh{i+1, j+1}-\rh{i+1,j})
      + \kappa_{i-1,j}^{c,n} (\rh{i-1, j+1}-\rh{i-1,j}),
  \end{align*}
  where
  \begin{align*}
    \kappa_{i,j}^{c,n} = \ &
                         \begin{cases}
                           \lambda_x \, \dfrac{H^{c,n}_{i+1/2,j+1}
                             (\rh{i,j+1}, \rh{i+1,j+1}) -
                             H^{c,n}_{i+1/2,j+1}(\rh{i,j}, \rh{i+1,j+1}) }
                           {\rh{i, j+1}-\rh{i,j}} & \mbox{ if } \rh{i,
                             j+1} \neq \rh{i,j},
                           \\
                           0 & \mbox{ if } \rh{i, j+1} = \rh{i,j},
                         \end{cases}
    \\
    \nu_{i,j}^{c,n} = \ &
                      \begin{cases}
                        - \lambda_x \, \dfrac{H^{c,n}_{i-1/2,j+1}
                          (\rh{i-1,j}, \rh{i,j+1}) - H^{c,n}_{i-1/2,j+1}
                          (\rh{i-1,j}, \rh{i,j}) } {\rh{i,
                            j+1}-\rh{i,j}} & \mbox{ if } \rh{i, j+1}
                        \neq \rh{i,j},
                        \\
                        0 & \mbox{ if } \rh{i, j+1} = \rh{i,j}.
                      \end{cases}
  \end{align*}
  As for $\delta_i^{c,n}$~\eqref{eq:deltainroe} and
  $\theta_i^{c,n}$~\eqref{eq:thetainroe}, it is immediate to prove that
  $\kappa_{i,j}^{c,n}, \, \nu_{i,j}^{c,n} \in \left[0, \dfrac12\right]$ for
  all $i, \, j \in \interi$. Hence,

  \begin{equation}
    \label{eq:DijOKroe}
    \sum_{i,j \in \interi} \modulo{\mathcal{D}_{i,j}^{c,n}} \leq
    \sum_{i,j \in \interi} \modulo{\rh{i,j+1} - \rh{i,j}}.
  \end{equation}

  Pass now to $\mathcal{E}_{i,j}^{c,n}$: we can proceed analogously to
  $\mathcal{B}_{i,j}^{c,n}$, treating separately the terms involving $V_1$
  and those involving $F^c$. First, by~\eqref{eq:fx1},

  \begin{align}
    \nonumber
    & V_1 (\x{i+1/2,j}, \rh{i,j}, \rh{i+1,j})
      - V_1 (\x{i+1/2,j+1}, \rh{i,j}, \rh{i+1,j})
     \\ \nonumber
    & + V_1 (\x{i-1/2,j+1}, \rh{i-1,j}, \rh{i,j})
      - V_1 (\x{i-1/2,j}, \rh{i-1,j}, \rh{i,j})
    \\ \nonumber
    = \
    &
      \vs_1 (x_{i+1/2,j}) \rh{i,j} + \min\left\{0, \vs_1 (x_{i+1/2,j})\right\} (\rh{i+1,j} - \rh{i,j})
    \\ \nonumber
    &- \vs_1 (x_{i+1/2,j+1}) \rh{i,j} - \min\left\{0, \vs_1 (x_{i+1/2,j+1})\right\} (\rh{i+1,j} - \rh{i,j})
    \\ \nonumber
    & + \vs_1 (x_{i-1/2,j+1}) \rh{i-1,j} + \min\left\{0, \vs_1 (x_{i-1/2,j+1})\right\} (\rh{i,j} - \rh{i-1,j})
    \\ \nonumber
    & -\vs_1 (x_{i-1/2,j}) \rh{i-1,j} - \min\left\{0, \vs_1 (x_{i-1/2,j})\right\} (\rh{i,j} - \rh{i-1,j})
    \\ \nonumber
    & \pm\left( \vs_1 (x_{i-1/2,j+1}) - \vs_1 (x_{i-1/2,j})\right)\rh{i,j}
    \\ \nonumber
    = \
    & \left(  \vs_1 (x_{i+1/2,j})  -  \vs_1 (x_{i+1/2,j+1})  -  \vs_1 (x_{i-1/2,j}) + \vs_1 (x_{i-1/2,j+1})
      \right) \rh{i,j}
    \\ \nonumber
    & + \left( \vs_1 (x_{i-1/2,j+1}) - \vs_1 (x_{i-1/2,j})\right) (\rh{i-1,j} - \rh{i,j})
    \\ \nonumber
    & + \left(
       \min\left\{0, \vs_1 (x_{i+1/2,j})\right\}-\min\left\{0, \vs_1 (x_{i+1/2,j+1})\right\}
      \right) (\rh{i+1,j} - \rh{i,j})
    \\ \nonumber
    &+  \left(
      \min\left\{0, \vs_1 (x_{i-1/2,j+1})\right\} -  \min\left\{0, \vs_1 (x_{i-1/2,j})\right\}
      \right)(\rh{i,j} - \rh{i-1,j})
    \\ \label{eq:12}
    \leq \
    & \dx \, \dy \, \norma{\partial_{xy}^2 \vs_1}_{\L\infty} \modulo{\rh{i,j}}
      + \dy \, \norma{\partial_y \vs_1}_{\L\infty}
      \left( \modulo{\rh{i+1,j} - \rh{i,j}}+  2 \, \modulo{\rh{i,j} - \rh{i-1,j}}
      \right),
  \end{align}
  since

  \begin{align*}
    & \vs_1 (x_{i+1/2,j})  -  \vs_1 (x_{i+1/2,j+1})  -  \vs_1 (x_{i-1/2,j}) + \vs_1 (x_{i-1/2,j+1})
     \\
    = \
    & \dx \, \partial_x \vs_1 (\xi_i,y_j) - \dx \, \partial_x \vs_1 (\xi_i, y_{j+1})
    \\
    = \
    & - \dx \, \dy \, \partial^2_{xy} \vs_1 (\xi_i, \zeta_{j+1/2}),
  \end{align*}
  with $\xi_i \in \ ]x_{i-1/2}, x_{i+1/2}[$ and $\zeta_{j+1/2} \in \ ] y_j, y_{j+1}[$.
  In a similar way, by~\eqref{eq:fx2},
  \begin{align}
    \nonumber
    & F^c (\rh{i,j}, \rh{i+1,j}, J_1^{c,n} (\x{i+1/2,j}))
      - F^c (\rh{i,j}, \rh{i+1,j}, J_1^{c,n} (\x{i+1/2,j+1}))
    \\ \nonumber
    & + F^c (\rh{i-1,j}, \rh{i,j}, J_1^{c,n} (\x{i-1/2,j+1}))
      - F^c (\rh{i-1,j}, \rh{i,j}, J_1^{c,n} (\x{i-1/2,j}))
    \\ \nonumber
    =\
    & J_1^{c,n} (\x{i+1/2,j}) \, \rh{i,j}
      + \min\left\{0, J_1^{c,n} (\x{i+1/2,j})\right\} \left(\rh{i+1,j} - \rh{i,j}\right)
    \\ \nonumber
    & - J_1^{c,n} (\x{i+1/2,j+1}) \, \rh{i,j}
      - \min\left\{0, J_1^{c,n} (\x{i+1/2,j+1})\right\} \left(\rh{i+1,j} - \rh{i,j}\right)
    \\ \nonumber
    &+  J_1^{c,n} (\x{i-1/2,j+1}) \, \rh{i-1,j}
      + \min\left\{0, J_1^{c,n} (\x{i-1/2,j+1})\right\} \left(\rh{i,j} - \rh{i-1,j}\right)
    \\ \nonumber
    & - J_1^{c,n} (\x{i-1/2,j}) \, \rh{i-1,j}
      - \min\left\{0, J_1^{c,n} (\x{i-1/2,j})\right\} \left(\rh{i,j} - \rh{i-1,j}\right)
    \\ \nonumber
    & \pm \left( J_1^{c,n} (\x{i-1/2,j+1})  - J_1^{c,n} (\x{i-1/2,j}) \right) \rh{i,j}
    \\ \nonumber
    = \
    & \left(
      J_1^{c,n} (\x{i+1/2,j})   - J_1^{c,n} (\x{i+1/2,j+1})  - J_1^{c,n} (\x{i-1/2,j}) + J_1^{c,n} (\x{i-1/2,j+1})
      \right)\rh{i,j}
    \\ \nonumber
    & + \left( J_1^{c,n} (\x{i-1/2,j+1})  - J_1^{c,n} (\x{i-1/2,j}) \right) \left(\rh{i-1,j} - \rh{i,j}\right)
    \\ \nonumber
    & + \left(
       \min\left\{0, J_1^{c,n} (\x{i+1/2,j})\right\}- \min\left\{0, J_1^{c,n} (\x{i+1/2,j+1})\right\}
      \right)\left(\rh{i+1,j} - \rh{i,j}\right)
    \\ \nonumber
    & + \left(
      \min\left\{0, J_1^{c,n} (\x{i-1/2,j+1})\right\} - \min\left\{0, J_1^{c,n} (\x{i-1/2,j})\right\}
      \right)\left(\rh{i,j} - \rh{i-1,j}\right)
    \\ \label{eq:13}
    \leq \
    & 
    2 \, \epsilon_c \, \dx \, \dy \, C_c \modulo{\rh{i,j}}
    \\
    \nonumber
      &+ \epsilon_c \dy \, \left(2\,\norma{\nabla^2 \eta_c}_{\L\infty} \!+ L_H\norma{\nabla\Tilde{\eta}_c}_{\L\infty}\right) \norma{r_{\Omega,o}}_{\L1}
      \!\!
      \left(
      \modulo{ \rh{i+1,j} - \rh{i,j}}
      + 2 \modulo{\rh{i,j} - \rh{i-1,j}}
      \right),
  \end{align}
  where we used~\eqref{eq:J1y}
  and~\eqref{eq:Jtriplabis}, with the notation~\eqref{eq:c12}.
  Therefore, collecting together~\eqref{eq:12} and~\eqref{eq:13}, we
  get
  \begin{align*}
    \modulo{\mathcal{E}_{i,j}^{c,n}} \leq \
    & \dx \, \dy \left( \norma{\partial_{xy}^2 \vs_1}_{\L\infty}
      +  2 \, \epsilon_c \,C_c \right) \modulo{\rh{i,j}}
    \\
    & +\dy \left(
      \norma{\partial_y \vs_1}_{\L\infty}
      +  \epsilon_c \left( 2\,\norma{\nabla^2 \eta_c}_{\L\infty} + L_H\norma{\nabla\Tilde{\eta}_c}_{\L\infty}\right)\norma{r_{\Omega,o}}_{\L1}
      \right)
      \\
      & \quad \quad\left(
       \modulo{ \rh{i+1,j} - \rh{i,j}}
      + 2 \modulo{\rh{i,j} - \rh{i-1,j}}
      \right),
  \end{align*}
  so that
  \begin{align}
    \nonumber
    \sum_{i,j \in \interi} \lambda_x  \modulo{\mathcal{E}_{i,j}^{c,n}}
    \leq \
    &  3 \lambda_x \dy \left(
      \norma{\partial_y \vs_1}_{\L\infty} \!\!\!
      + \! \epsilon_c\left(\! 2 \norma{\nabla^2 \eta_c}_{\L\infty}\!\!\!+\!L_H\norma{\nabla\Tilde{\eta}_c}_{\L\infty}\right)\norma{r_{\Omega,o}}_{\L1}
      \!\right)
    \\
    \label{eq:EijOKroe}
    & \cdot \sum_{i,j \in \interi} \modulo{\rh{i+1,j} - \rh{i,j}} \,
    + \dt   \left( \norma{\partial_{xy}^2 \vs_1}_{\L\infty}
      +  2 \, \epsilon_c \,C_c \right)  \, \dy  \sum_{i,j \in \interi} \rh{i,j}.
  \end{align}
  Hence, by~\eqref{eq:DijOKroe} and~\eqref{eq:EijOKroe}, using also
  Lemma~\ref{lem:L1roe}, we obtain
  \begin{align}
    \nonumber
    & \sum_{i,j \in \interi} \dx \, \modulo{\rho^{c,n+1/2}_{i,j+1} - \rho^{c,n+1/2}_{i,j}}
    \\\nonumber
    \leq \
    & \sum_{i,j \in \interi} \dx \left( \modulo{\mathcal{D}_{i,j}^{c,n}}
      + \lambda_x \modulo{\mathcal{E}_{i,j}^{c,n}}\right)
    \\
    \label{eq:BVdxroe}
    \leq \
    & \sum_{i,j \in \interi} \dx \, \modulo{\rh{i,j+1} - \rh{i,j}} + \dt \left(
      \norma{\partial_{xy}^2 \vs_1}_{\L\infty} + 2 \,  \epsilon_c  \, C_c  \right) \norma{\rho^c_o}_{\L1}.
    \\
    \nonumber
    & + 3 \, \dt  \left(
      \norma{\partial_y \vs_1}_{\L\infty}
      +  \epsilon_c \left( 2  \norma{\nabla^2 \eta_c}_{\L\infty} \!+ L_H\norma{\nabla\Tilde{\eta}_c}_{\L\infty}\right) \norma{r_{\Omega,o}}_{\L1}
      \right)
      \sum_{i,j \in \interi} \dy \, \modulo{\rh{i+1,j} - \rh{i,j}}
  \end{align}
  Setting
  \begin{align}
    \label{eq:K1roe}
    K^c_1 = \
    & 3 \left(\norma{\partial_x \vs_1}_{\L\infty}
      + \norma{\partial_y \vs_1}_{\L\infty}
      +  \epsilon_c \, \left( 4 \,
      \norma{\nabla^2 \eta_c}_{\L\infty}+2L_H\norma{\nabla\Tilde{\eta}_c}_{\L\infty}\right) \norma{r_{\Omega,o}}_{\L1}\right) ,
    \\
    \label{eq:K2roe}
    K^c_2 = \
    &  \left( 4 \, \epsilon_c   \, C_c
      + 2 \, \norma{\partial_{xx}^2 \vs_1}_{\L\infty} + \norma{\partial_{xy}^2 \vs_1}_{\L\infty}
      \right)
      \norma{\rho^c_o}_{\L1},
  \end{align}
  by~\eqref{eq:BVdyroe} and~\eqref{eq:BVdxroe} we conclude
  \begin{align*}
    & \sum_{i,j \in \interi} \left(
      \dy \, \modulo{\rho^{c,n+1/2}_{i+1,j} - \rho^{c,n+1/2}_{i,j}}
      +
      \dx \, \modulo{\rho^{c,n+1/2}_{i,j+1} - \rho^{c,n+1/2}_{i,j}}
      \right)
    \\
    \leq \
    & \left( 1 + \dt  K^c_1 \right)
      \sum_{i,j \in \interi} \left(
      \dx \, \modulo{\rh{i,j+1} - \rh{i,j}} +  \dy \, \modulo{\rh{i+1,j} -\rh{i,j}}
      \right)
      +  \dt \, K^c_2.
  \end{align*}
  Analogous computations yield
  \begin{align*}
    & \sum_{i,j \in \interi} \left(
      \dy \, \modulo{\rho^{c,n+1}_{i+1,j} - \rho^{c,n+1}_{i,j}}
      +
      \dx \, \modulo{\rho^{c,n+1}_{i,j+1} - \rho^{c,n+1}_{i,j}}
      \right)
    \\
    \leq \
    & \left( 1 + \dt \, K^c_3 \right)
      \sum_{i,j \in \interi} \left(
      \dx \, \modulo{\rho^{c,n+1/2}_{i,j+1} - \rho^{c,n+1/2}_{i,j}}
      +  \dy \, \modulo{\rho^{c,n+1/2}_{i+1,j} -\rho^{c,n+1/2}_{i,j}}
      \right)
      +  \dt \, K^c_4,
  \end{align*}
  where
  \begin{align}
    \label{eq:K3roe }
    K^c_3 = \
    & 3 \left( \norma{\partial_x \vs_2}_{\L\infty}
      +  \norma{\partial_y \vs_2}_{\L\infty}
      + \epsilon_c \, \left( 4
      \norma{\nabla^2 \eta_c}_{\L\infty} + 2 L_H\,\norma{\nabla\Tilde{\eta}_c}_{\L\infty}\right) \norma{r_{\Omega,o}}_{\L1}\right),
    \\
    \label{eq:K4roe}
    K^c_4 = \
    &  \left(
      4 \, \epsilon_c  \, C_c
      + 2 \, \norma{\partial_{yy}^2 \vs_2}_{\L\infty} + \norma{\partial_{xy}^2 \vs_2}_{\L\infty}
      \right) \norma{\rho^c_o}_{\L1}.
  \end{align}
  Observe that, using the notation~\eqref{eq:K1defroe}
  and~\eqref{eq:K2defroe},
  \begin{align*}
    K^c_1, \, K^c_3 \leq \mathcal{K}^c_1,
    \qquad
    K^c_2, \, K^c_4 \leq \mathcal{K}^c_2.
  \end{align*}
  A recursive argument yields the desired result:
  \begin{align*}
    & \sum_{i,j \in \interi} \left(
      \dy \, \modulo{\rho^{c,n}_{i+1,j} - \rho^{c,n}_{i,j}}
      +
      \dx \, \modulo{\rho^{c,n}_{i,j+1} - \rho^{c,n}_{i,j}}
      \right)
    \\
    \leq \
    & e^{2 \, n \, \dt \, \mathcal{K}^c_1}  \sum_{i,j \in \interi} \left(
      \dx \, \modulo{\rho^{c,0}_{i,j+1} - \rho^{c,0}_{i,j}}
      +  \dy \, \modulo{\rho^{c,0}_{i+1,j} -\rho^{c,0}_{i,j}}
      \right)
      +  \frac{2 \, \mathcal{K}^c_2}{\mathcal{K}^c_1}\left(e^{2 \, n \, \dt \, \mathcal{K}^c_1} -1
      \right).
  \end{align*}
  
\end{proof}

\begin{corollary} {\bf ($\BV$ estimate in space and time)}
  \label{cor:bvxtroe}
  Let $\brho_o \in (\L1\cap\L\infty \cap \BV) (\Omega;
  \reali^N_+)$. Let~\ref{vs}, \ref{H}, \ref{eta}, \eqref{eq:CFLroe}
  hold. Then, for all $c\in\{1,\dots,N\}$ and $t>0$, $\rho^c_\Delta$ 
  constructed through~\eqref{eq:scheme1roe} and~\eqref{eq:scheme2roe}  
  satisfies the following
  estimate: for all $n=1, \ldots, N_T$,
  \begin{equation}
    \label{eq:bvxt}
    \sum_{m=0}^{n-1} \sum_{i,j \in \interi}\!
    \dt\! \left(
      \dy  \modulo{\rho^{c,m}_{i+1,j} - \rho^{c,m}_{i,j}}
      + \!\dx \modulo{\rho^{c,m}_{i,j+1} - \rho^{c,m}_{i,j}}\right)
    +
    \sum_{m=0}^{n-1} \sum_{i,j \in \interi}
    \!\dx \dy  \modulo{\rho^{c,m+1}_{i,j} - \rho^{c,m}_{i,j}}
    \leq \mathcal{C}^c_{xt}(t^n)
  \end{equation}
  where
  \begin{equation}
    \label{eq:Cxt}
    \mathcal{C}^c_{xt} (t) = t \left(\mathcal{C}^c_x (t) + 2 \, \mathcal{C}^c_t (t)\right),
  \end{equation}
  with $\mathcal{C}^c_x$ as in~\eqref{eq:Cx} and $\mathcal{C}^c_t$ as in~\eqref{eq:Ct}.
\end{corollary}

\begin{proof}
  By Proposition~\ref{prop:bvroe} we have
  \begin{equation}
    \label{eq:9r}
    \sum_{m=0}^{n-1} \sum_{i,j \in \interi}
    \dt \left(
      \dy \, \modulo{\rho^{c,m}_{i+1,j} - \rho^{c,m}_{i,j}}
      + \dx \, \modulo{\rho^{c,m}_{i,j+1} - \rho^{c,m}_{i,j}}\right)
    \leq
    n \, \dt \, \mathcal{C}^c_x (n \, \dt).
  \end{equation}
  Since
  \begin{displaymath}
    \modulo{\rho^{c,m+1}_{i,j} - \rho^{c,m}_{i,j}}
    \leq
     \modulo{\rho^{c,m+1}_{i,j} - \rho^{c,m+1/2}_{i,j}}
     +
     \modulo{\rho^{c,m+1/2}_{i,j} - \rho^{c,m}_{i,j}},
  \end{displaymath}
  we focus first on
  \begin{displaymath}
    \sum_{i,j \in \interi}
    \dx \, \dy \,\modulo{\rho^{c,m+1/2}_{i,j} - \rho^{c,m}_{i,j}}.
  \end{displaymath}
  By the scheme~\eqref{eq:scheme1roe}, we have, using the notation~\eqref{eq:notvJ},
  \begin{align*}
    \rho^{c,m+1/2}_{i,j} - \rho^{c,m}_{i,j} = \
    & - \lambda_x \left[
      V_1 (x_{i+1/2,j}, \rho^{c,m}_{i,j}, \rho^{c,m}_{i+1,j}) + F^c (\rho^{c,m}_{i,j}, \rho^{c,m}_{i+1,j}, J^{c,m}_1 (x_{i+1/2,j}))\right.
    \\
    & \qquad \left.
      - V_1 (x_{i-1/2,j}, \rho^{k,m}_{i-1,j}, \rho^{c,m}_{i,j}) - F^c (\rho^{c,m}_{i-1,j}, \rho^{c,m}_{i,j}, J^{c,m}_1 (x_{i-1/2,j}))
      \right]
    \\
    = \
    &  - \lambda_x \left[
      v_{i+1/2,j} \, \rho^{c,m}_{i,j} + \min\left\{0, v_{i+1/2,j}\right\} (\rho^{c,m}_{i+1,j} - \rho^{c,m}_{i,j})\right.
    \\
    & -  v_{i-1/2,j} \, \rho^{c,m}_{i-1,j} - \min\left\{0, v_{i-1/2,j}\right\} (\rho^{c,m}_{i,j} - \rho^{c,m}_{i-1,j})
    \\
    & + J_1^{c,m}(x_{i+1/2,j}) \, \rho^{c,m}_{i,j} + \min\left\{0, J^{c,m}_1(x_{i+1/2,j})\right\}
      (\rho^{c,m}_{i+1,j} - \rho^{c,m}_{i,j})
    \\
    & -  J^{c,m}_1(x_{i-1/2,j}) \, \rho^{c,m}_{i-1,j} - \min\left\{0, J^{c,m}_1(x_{i-1/2,j})\right\}
      (\rho^{c,m}_{i,j} - \rho^{c,m}_{i-1,j})
    \\
    & \left.
      \pm v_{i-1/2,j} \, \rho^{c,m}_{i,j} \pm J^{c,m}_1(x_{i-1/2,j})\, \rho^{c,m}_{i,j}
      \right]
    \\
    = \
    & - \lambda_x \left[
      \dx \, \partial_x \vs_1 (\xi_i, y_j) \, \rho^{c,m}_{i,j}
      \right.
    \\
    & + \left(v_{i-1/2,j} -  \min\left\{0, v_{i-1/2,j}\right\} \right)(\rho^{c,m}_{i,j} - \rho^{c,m}_{i-1,j})
    \\
    & + \min\left\{0, v_{i+1/2,j}\right\} (\rho^{c,m}_{i+1,j} - \rho^{c,m}_{i,j})
    \\
    & + \left( J_1^{c,m}(x_{i+1/2,j}) - J^{c,m}_1(x_{i-1/2,j})\right)\, \rho^{c,m}_{i,j}
    \\
    & + \left(
      J^{c,m}_1(x_{i-1/2,j}) - \min\left\{0, J^{c,m}_1(x_{i-1/2,j})\right\}
      \right) (\rho^{c,m}_{i,j} - \rho^{c,m}_{i-1,j})
      \\
    & \left.+ \min\left\{0, J^{c,m}_1(x_{i+1/2,j})\right\}
      (\rho^{c,m}_{i+1,j} - \rho^{c,m}_{i,j})\right]
    \\
    \leq \
    &
      \dt \left(
      \norma{\partial_x \vs_1}_{\L\infty}
      + \epsilon_c \,\left(2\,\norma{\nabla^2 \eta_c}_{\L\infty}+ L_H\,\norma{\nabla\Tilde{\eta}_c}_{\L\infty} \right) \norma{r_{\Omega,o}}_{\L1}
      \right) \rho^{c,m}_{i,j}
    \\
    & + \lambda_x \left( \norma{\vs_1}_{\L\infty} + \epsilon_c  \right)
      \left( \modulo{\rho^{c,m}_{i,j} - \rho^{c,m}_{i-1,j}} + \modulo{\rho^{c,m}_{i+1,j} -\rho^{c,m}_{i,j}}\right),
  \end{align*}
  where $\xi_i \in \ ]x_{i-1/2}, x_{i+1/2}[$ and we used
  \eqref{eq:Jinf} and~\eqref{eq:Jx}.  Therefore,
  \begin{align*}
     \sum_{i,j \in \interi}
    \dx  \dy  & \modulo{\rho^{c,m+1/2}_{i,j} - \rho^{c,m}_{i,j}} 
    \\
    & \leq \
    \dt \left(
      \norma{\partial_x \vs_1}_{\L\infty} \!\!
      +  \epsilon_c \left(2 \norma{\nabla^2 \eta_c}_{\L\infty} \!\!+\! L_H\norma{\nabla\Tilde{\eta}_c}_{\L\infty}\right) \norma{r_{\Omega,o}}_{\L1}
    \right)\!
     \norma{\rho^c_o}_{\L1}
    \\
    & \quad+ 2 \, \dt \,  \left(\norma{\vs_1}_{\L\infty} + \epsilon_c  \right)
    \sum_{i,j \in \interi} \dy \, \modulo{\rho^{c,m}_{i+1,j} - \rho^{c,m}_{i,j}}    
    \\
    & \leq \
      \dt \left(
      \norma{\partial_x \vs_1}_{\L\infty}\!\!
      + \epsilon_c \left(2  \norma{\nabla^2 \eta_c}_{\L\infty} \!\!+ \!L_H\norma{\nabla\Tilde{\eta}_c}_{\L\infty}\right) \norma{r_{\Omega,o}}_{\L1}
    \right) \!\norma{\rho^c_o}_{\L1}
    \\
    &  \quad+ 2 \, \dt \,  \left( \norma{\vs_1}_{\L\infty} + \epsilon_c   \right)
    \mathcal{C}^c_x (m \, \dt)
    \\
    &\leq \
     \dt \, \mathcal{C}^c_t (m \, \dt),
  \end{align*}
  where we set
  \begin{align}
    \nonumber
    \mathcal{C}^c_t (s) = &
    2 \,\left(\norma{\boldsymbol{\vs}}_{\L\infty}
      + \epsilon_c  \right)\,\mathcal{C}^c_x (s)
      \\
      \nonumber
    & +  \left(
      \norma{\nabla \boldsymbol{\vs}}_{\L\infty}
      + \epsilon_c \left(2  \norma{\nabla^2 \eta_c}_{\L\infty} + L_H\norma{\nabla\Tilde{\eta}_c}_{\L\infty}\right)\norma{r_{\Omega,o}}_{\L1}
    \right)\! \norma{\rho^c_o}_{\L1}
    \\
    \label{eq:Ct}
      \mathcal{C}_t (s) = &\sum_{c=1}^N \mathcal{C}^c_t (s).
  \end{align}
  Analogously, we get
  \begin{align*}
    \sum_{i,j \in \interi}
    \dx  \dy & \modulo{\rho^{c,m+1}_{i,j} - \rho^{c,m+1/2}_{i,j}}
    \\
    & \leq \
    \dt \left(
      \norma{\partial_y \vs_2}_{\L\infty}
      \!\!+ \epsilon_c \left(2  \norma{\nabla^2 \eta_c}_{\L\infty} \!\!+\! L_H\norma{\nabla\Tilde{\eta}_c}_{\L\infty}\right) \norma{r_{\Omega,o}}_{\L1}
    \right) \!\norma{\rho^c_o}_{\L1}
    \\
    & \quad + 2\,  \dt \,  \left( \norma{\vs_2}_{\L\infty} + \epsilon_c \right)
    \sum_{i,j \in \interi} \dx \, \modulo{\rho^{c,m+1/2}_{i,j+1} - \rho^{c,m+1/2}_{i,j}}
    \\
    & \leq \
      \dt \left(
      \norma{\partial_y \vs_2}_{\L\infty}
      \!\!+ \epsilon_c \left(2  \norma{\nabla^2 \eta_c}_{\L\infty} \!\!+\! L_H\norma{\nabla\Tilde{\eta}_c}_{\L\infty}\right) \norma{r_{\Omega,o}}_{\L1}
    \right)\! \norma{\rho^c_o}_{\L1}
    \\
    & \quad + 2\,   \dt \,  \left(\norma{\vs_2}_{\L\infty} + \epsilon_c \right)
    \mathcal{C}^c_x ((m+1) \, \dt)
    \\
    & \leq \
     \dt \, \mathcal{C}^c_t ((m+1) \, \dt).
  \end{align*}
  Hence
  \begin{equation}
    \label{eq:8r}
    \sum_{m=0}^{n-1} \sum_{i,j \in \interi}
    \dx \, \dy \, \modulo{\rho^{c,m+1}_{i,j} - \rho^{c,m}_{i,j}}
    \leq
     2 \, n \, \dt \, \mathcal{C}^c_t (n \, \dt),
  \end{equation}
  which, together with~\eqref{eq:9r}, completes the proof.
\end{proof}

\subsection{Lipschitz continuous dependence on initial data}
\label{sec:lipdep}

\begin{proposition}
  \label{prop:lipdep}
  Fix $T>0$. Let~\ref{Omega1}, \ref{Omega2}, \ref{vs}, \ref{H} and \ref{eta} hold. Let
  $\brho_o, \, \boldsymbol{\sigma}_o \in (\L1\cap \L\infty \cap \BV) (\Omega;
  \reali^N_+)$. Call $\brho$ and $\boldsymbol{\sigma}$ the corresponding weak solutions to~\eqref{eq:1}. Then the following estimate holds:
  \begin{displaymath}
     \norma{\brho (t) - \boldsymbol{\sigma} (t)}_{\L1 (\Omega;\R^N)} \leq \mathcal{K}(t)
     \norma{\brho_o - \boldsymbol{\sigma}_o}_{\L1 (\Omega;\R^N)} 
  \end{displaymath}
  where $\mathcal{K}(t)$ as defined in~\cite[Theorem 2]{GoatinRossi2024} is given by 
  \begin{displaymath}
      \mathcal{K}(t) = e^{N\int_0^t\Tilde{\mathcal{K}}(s)ds}.
  \end{displaymath} 
  where $\Tilde{\mathcal{K}}(t)$ depends on $\epsilon,\|\nabla \eta_c\|_\L\infty, \|\nabla^2 \eta_c\|_\L\infty, L_H, L_{H'}, \|\nabla \Tilde{\eta}\|_\L\infty$, $\|\nabla^2 \Tilde{\eta}\|_\L\infty$, $\|\boldsymbol{\vs}\|_{W^{2,\infty}}$,
  \\
  $\|\brho_o\|_\L1$ and $TV(\brho_o)$.
\end{proposition}

\begin{proof}
For this proof we rely on a result in \cite{GoatinRossi2024}, where the stability with respect to the initial data was shown under less restrictive assumptions. To show the Lipschitz estimate with respect to the initial data it is only necessary to verify the assumptions $(\nu)$ and $(\mathcal{J})$ of \cite[Theorem 2]{GoatinRossi2024}. In our setting it follows by assumptions \ref{vs},\ref{H} and \ref{eta} that $\boldsymbol{\vs+v_c^{dyn}}\in \C2(\reali^2,\reali^2)$ and $\|\boldsymbol{(\vs+v_c^{dyn})}(t,\cdot,\cdot)\|_{\C2}$ is uniformly bounded for all $t\geq0$ and for all $c=1,...,N$. Since $\eta_c,\Tilde{\eta}_c \in (\C3 \cap \W3\infty )(\reali^2; \reali^+)$ for all $c\in\{1,\dots,N\}$ the nonlocal operator $\mathcal{J}^c = \boldsymbol{v^{dyn}_c}:\L1(\reali^2;\reali^{N+1})\rightarrow \C2(\reali^2;\reali^{2})$. 
Setting $\nu^c(t,x,A)=\boldsymbol{\vs}(x)+A$ the assumptions ($\boldsymbol\nu$) and  ($\mathcal{J}$.1)--($\mathcal{J}$.2) in \cite[Theorem 2]{GoatinRossi2024} follow then from \ref{vs}, \ref{H} and \ref{eta} and the fact that
$z\mapsto z/\sqrt{1+|z|^2}$ is smooth with bounded derivatives.
Thus, the assumptions of \cite[Theorem 2]{GoatinRossi2024} are met and the Lipschitz continuity with respect to the initial data follows.
Since weak solutions are unique \cite{KPS2018} and therefore coincide with the weak entropy solution, the result follows.
\end{proof}

\section{Numerical results}
\label{sec:results}

We follow the test setting given in \cite{Rossietal} for two classes of objects with different parameters $\sigma_c$, $c\in\{1,2\}$, using the Roe scheme stated in~\eqref{eq:scheme1roe} and~\eqref{eq:scheme2roe}.


\subsection{Test setting}

We simulate a setting where particles in the shape of metal cylinders are transported on a conveyor belt moving with speed $\| \boldsymbol{\vs}\|=v_T = 0.1$ m/s and are redirected by a diverter which is positioned at an angle \(\theta = 55\) degrees.
The static velocity field is given by $\boldsymbol{\vs}=(0.1\;0)^T$. Figure 1 provides an illustration of the conveyor belt and the obstacle. 
Setting the mass of the diverter (Region C) high enough ($>r_{max}$) prohibits parts from passing through the diverter, see~\cite{burger2020} and \cite{goatin2023}. 

\begin{figure}[H]
\centering
\setlength{\abovecaptionskip}{2pt}
\begin{tikzpicture}
\node[anchor=south west,inner sep=0] (Bild) at (0,0)
{\includegraphics[width=0.4\textwidth]{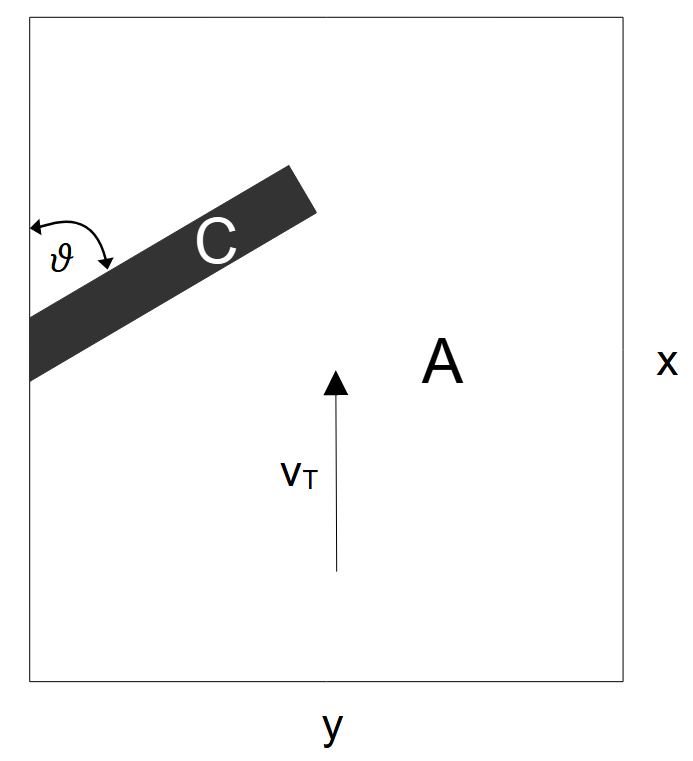}};
\begin{scope}[x=(Bild.south east),y=(Bild.north west)]
	\node (E5) at (0.2,0.4) {};
\node (E6) at (0.5,0.4) {};


\end{scope}
\end{tikzpicture}
\caption{Schematic view of the static field of the conveyor belt.}
\label{img:staticvelocityfield}
\setlength{\abovecaptionskip}{10pt}
\end{figure}

\subsection{Discretisation and solution properties}

For the numerical model, we introduce a uniform cartesian grid with \(\Delta x=\Delta y\) on the selected area of the conveyor belt given by $\Omega=[0,0.7]\times[0,0.8]$ with $N_x=141$ and $N_y=161$ cells. Initial conditions for the density at time \(t=0\) are given by experimental data, where a total of $N_p=192$ particles were placed on the conveyor belt in the area $\Omega_0 = [0,0.4]\times[0,0.6]$. For given initial positions of the cargo $(\mathbf{x}_{i}^0)_{i\leq N_n}\subset\Omega_0$ the initial density is given by 
\begin{align*}
    \rho_0(\mathbf{x}) = \frac{\gamma}{2\pi\rho_{\max}}\sum_{i=1}^{N} e^{-\frac{1}{2}\gamma\|\mathbf{x}-\mathbf{x}_{i}^0\|_2^2},\qquad \gamma = 0.02, \quad \rho_{\max}= 2004.
\end{align*}
following \cite{original}. We then divide the initial mass vertically at $x=1/3$ into two separate species $c\in\{1,2\}$, where we assign each population a different \(\sigma_c\). The PCE coefficients $\alpha_c$ are set to $\alpha_1=1$ and $\alpha_2=2$. The mollifier $\eta_c$, which is used in the operator $\boldsymbol{I}_c[\rho]$ given in~\eqref{eq:vdyn}, is chosen as follows
\begin{align*}
\eta_c(x) = \frac{\sigma_c}{2 \pi} e^{-\frac{1}{2} \sigma_c \vert \vert x \vert \vert_2^2},\quad c=1,2,
\end{align*}
with $\sigma_1 = 10000$ and $\sigma_2 = \frac{1}{2}\sigma_1 = 5 000$. 
The mollifier $\tilde{\eta}_c$, which is used in the Heaviside function $H$, see~\eqref{eq:vdyn}, is defined in the same way:
\begin{align*}
\tilde{\eta}_c(x) = \frac{\sigma_3}{2 \pi} e^{-\frac{1}{2} \sigma_3 \vert \vert x \vert \vert_2^2},
\end{align*}
where we take $\sigma_3 = \sigma_1$ and $\tilde{\eta_1}=\tilde{\eta_2}$.
In the original model formulation~\cite{original}, the Heaviside function was introduced to avoid densities larger than $r_{\max}$. In this work we use the approximation $H_t$ (atan), obtained using the inverse tangent. 

\begin{equation}
\label{eq:Ht}
    H_t(u) = \frac{\arctan(50 (u-1))}{\pi} + 0.5.
\end{equation}
where we rescale $u= \frac{\tilde{\eta}*r_\Omega}{r_{\max}}$.
Using the inverse tangent approximation corresponds to activating the collision operator $\boldsymbol{I}_c[\rho]$ very close to $r_{\max}$, see \cite{Rossietal}. Due to the smoothing of $H$ and the introduction of the mollifier $\Tilde{\eta}$, the maximal density constraint is not satisfied. Thus, there is no maximum principle.
\\
\\
For our comparison with a microscopic simulation, we employ a collision detection engine based on the results on material flow developed in~\cite{original}.
We consider two configurations. First we simulate the case where the smaller particle class $c=1$ (blue) is placed downstream of the larger particle class $c=2$ (red). Placing the larger objects behind the smaller ones, the smaller particles arrive first at the obstacle. As it can be seen in the top row of Figure~\ref{Figure2}, the smaller particles are pushed by the incoming larger ones to the edge of the obstacle. We can also observe that the particles do not mix in the outflow, with the smaller particles (blue) staying on the left side. We observe the same effect in the microscopic simulation in Figure~\ref{Figure2}, bottom.

\begin{figure}[htb]
    \centering 
\begin{subfigure}{0.25\textwidth}
  
  \reflectbox{\includegraphics[angle=90,width=\linewidth]{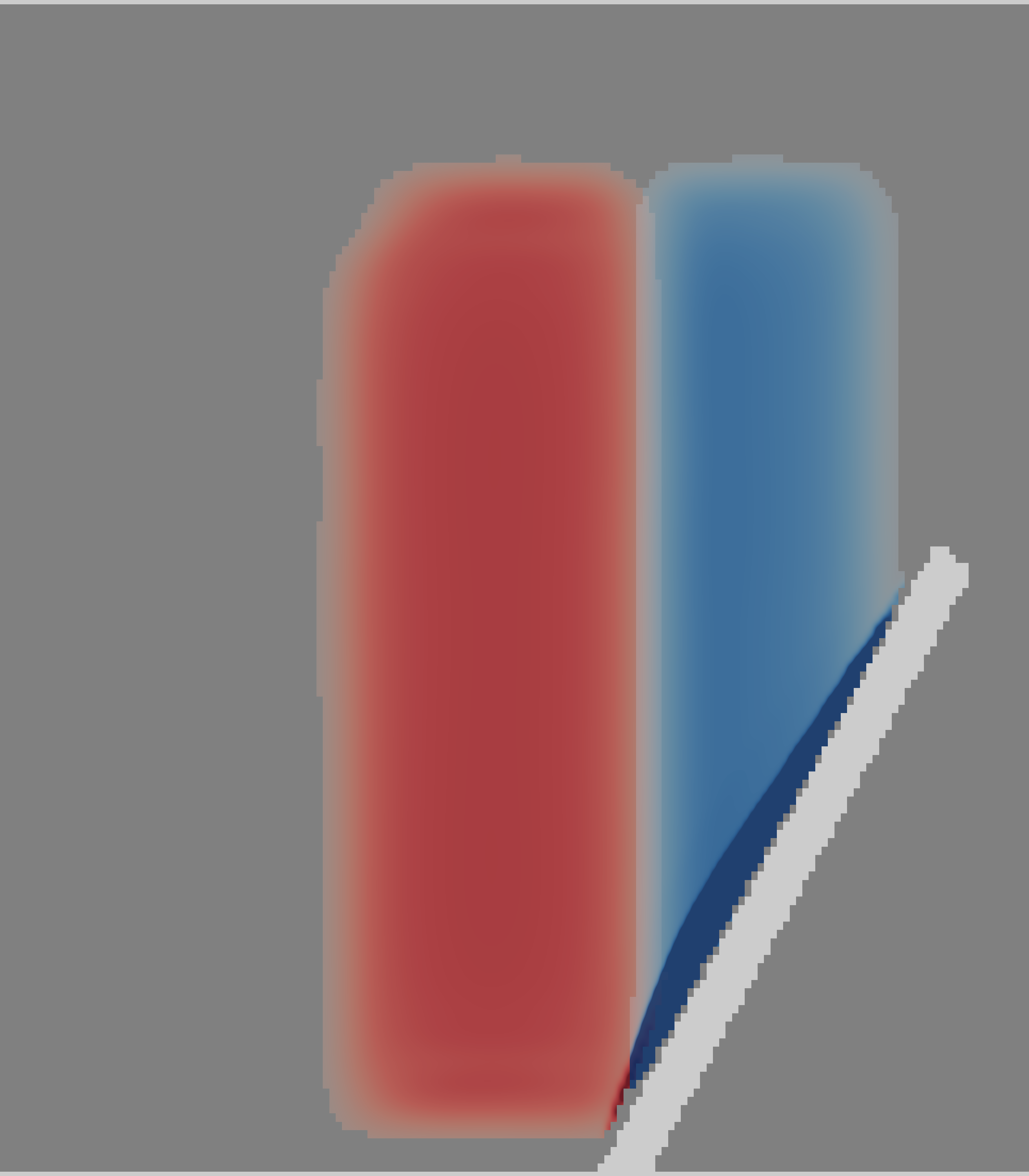}}
  \caption{time $t_1$}
  \label{fig:1}
\end{subfigure}\hfil 
\begin{subfigure}{0.25\textwidth}
  \reflectbox{\includegraphics[angle=90,width=\linewidth]{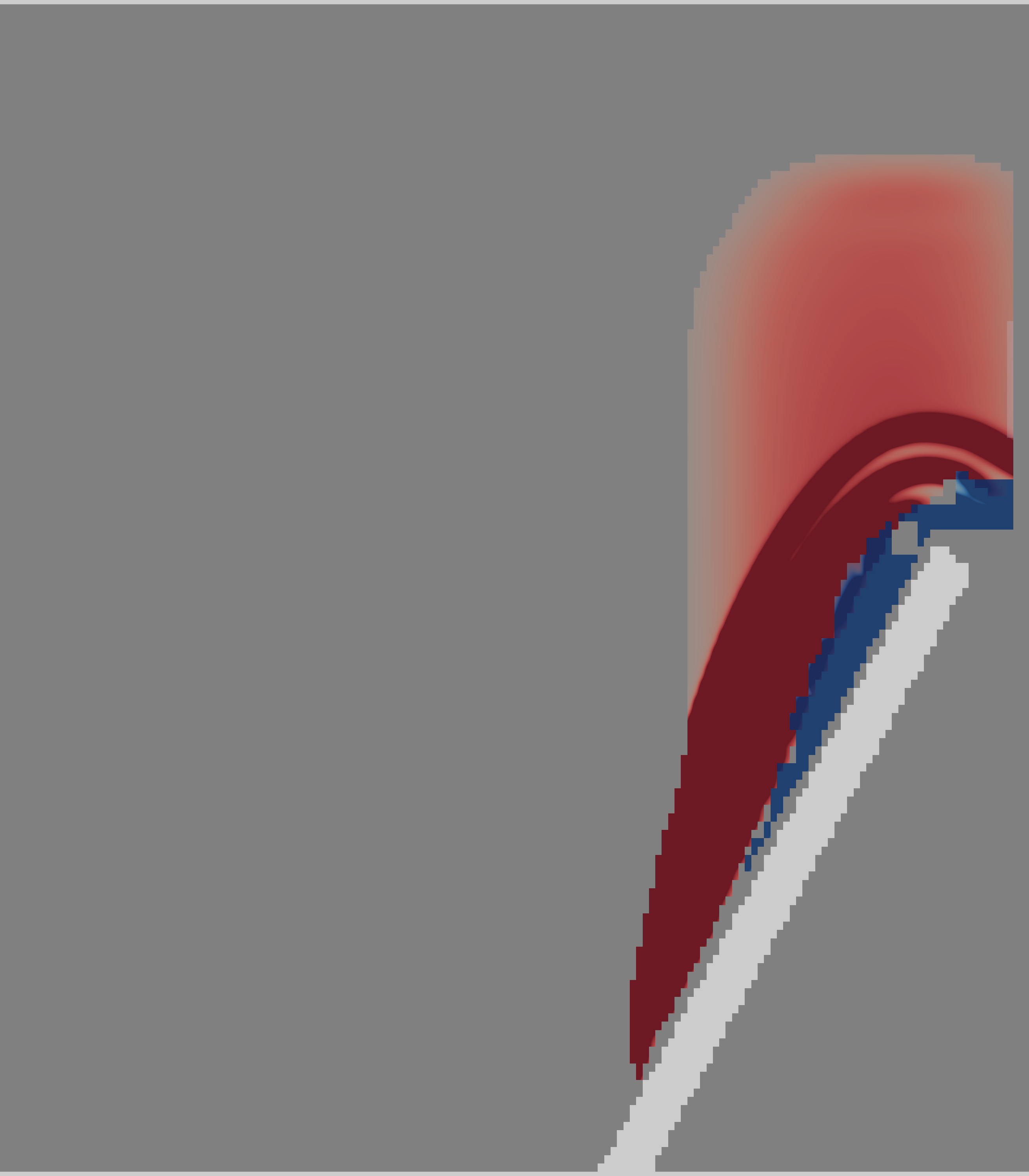}}
  \caption{time $t_2$}
  \label{fig:2}
\end{subfigure}\hfil 
\begin{subfigure}{0.25\textwidth}
  \reflectbox{\includegraphics[angle=90,width=\linewidth]{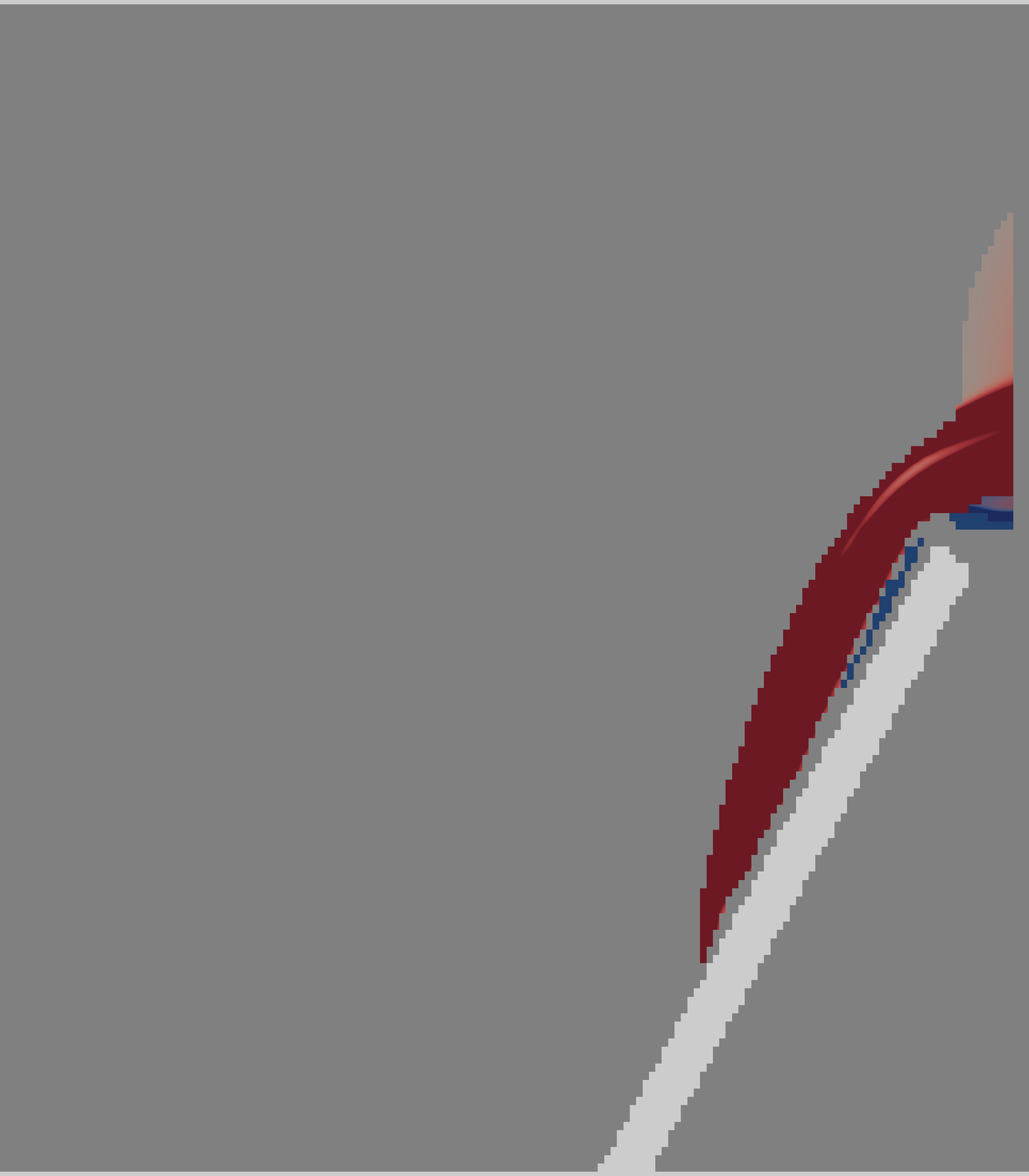}}
  \caption{time $t_3$}
  \label{fig:3}
\end{subfigure}

\medskip

\begin{subfigure}{0.25\textwidth}
  \includegraphics[width=\linewidth]{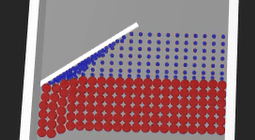}
  \label{fig:4}
\end{subfigure}\hfil 
\begin{subfigure}{0.25\textwidth}
  \includegraphics[width=\linewidth]{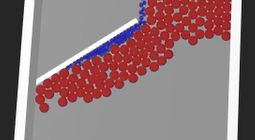}
  \label{fig:5}
\end{subfigure}\hfil 
\begin{subfigure}{0.25\textwidth}
  \includegraphics[width=\linewidth]{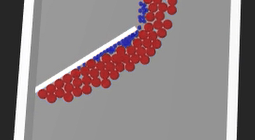}
  \label{fig:6}
\end{subfigure}
\caption{The smaller particles (blue) start first and are pushed to the edge of the obstacle by the coming larger particles (red).}
\label{Figure2}
\end{figure}

In the second case we swap the classes positions. The macroscopic simulation shows that the incoming smaller particles penetrate the mass of the downstream larger particles until they reach the diverter, see Figure~\ref{Figure3}. The different particle classes mix in the outflow with the smaller particles (blue) overtaking some of the larger ones. This effect can also be seen in the outflow diagram depicted in Figure~\ref{Figure5}. The same effect can be observed in the microscopic simulation, see Figure~\ref{Figure3}, bottom.

\begin{figure}[htb]
    \centering 
\begin{subfigure}{0.25\textwidth}
  \reflectbox{\includegraphics[angle=90,width=\linewidth]{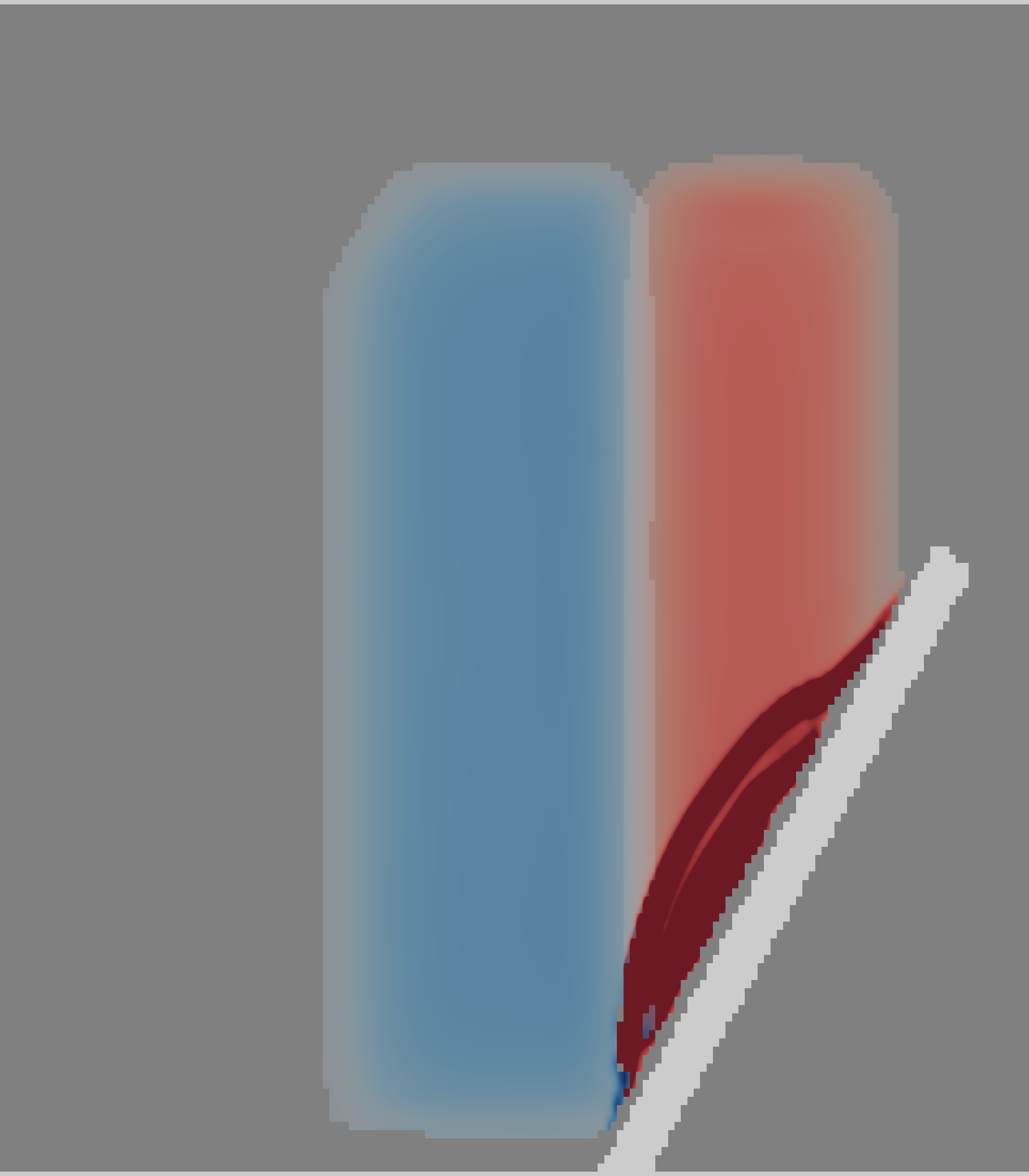}}
  \caption{time $t_1$}
  \label{fig:1}
\end{subfigure}\hfil 
\begin{subfigure}{0.25\textwidth}
  \reflectbox{\includegraphics[angle=90,width=\linewidth]{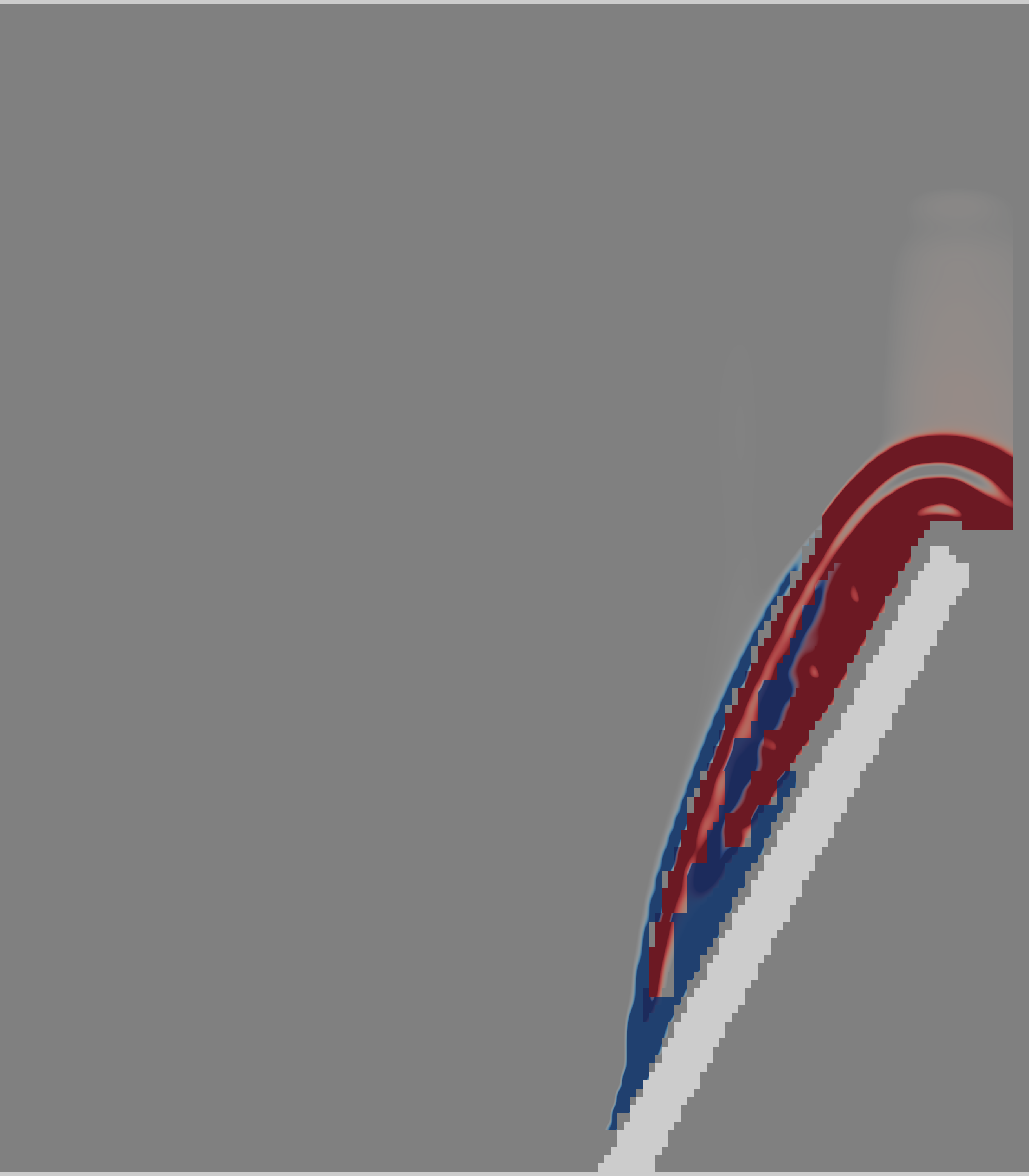}}
  \caption{time $t_2$}
  \label{fig:2}
\end{subfigure}\hfil 
\begin{subfigure}{0.25\textwidth}
  \reflectbox{\includegraphics[angle=90,width=\linewidth]{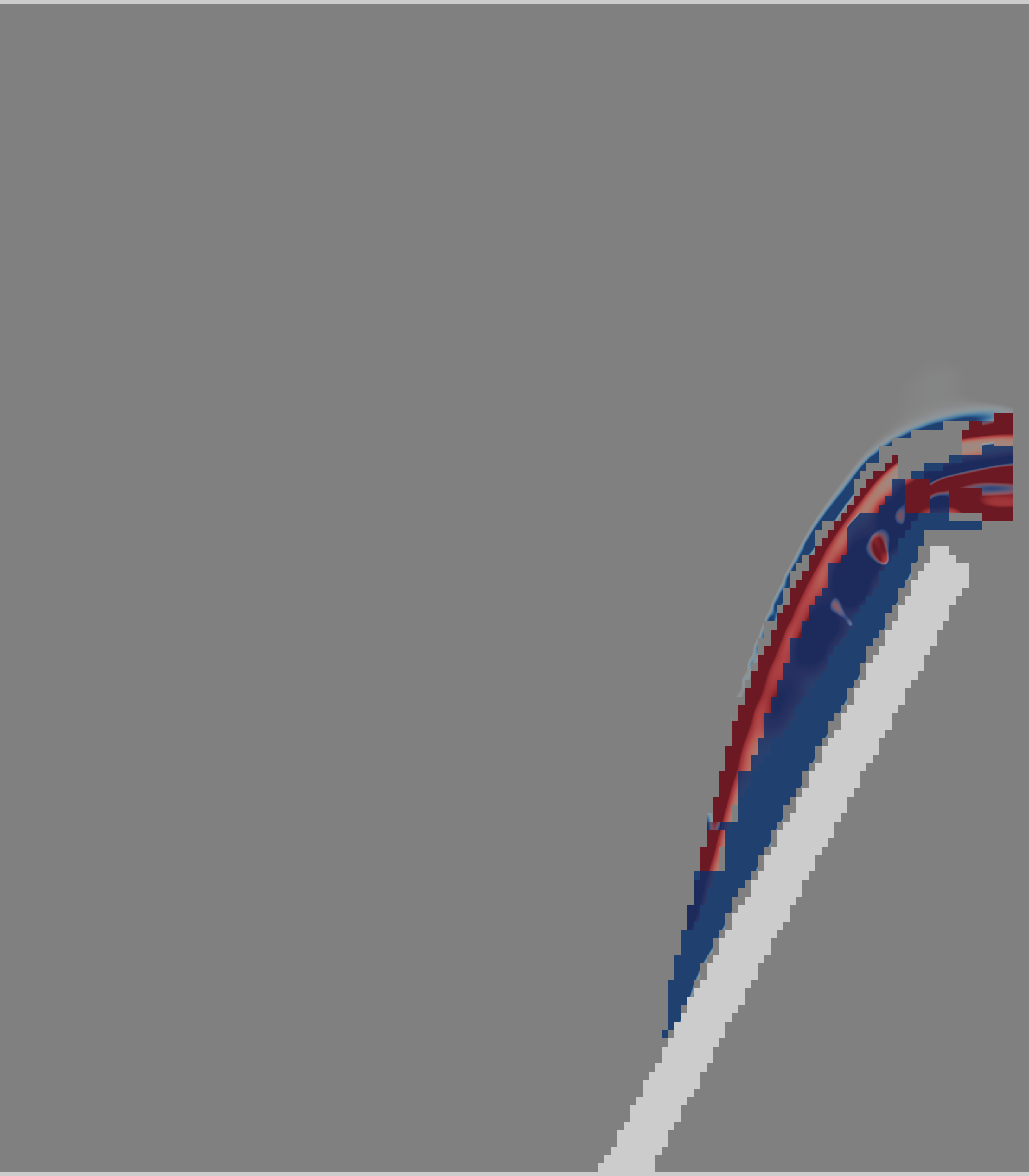}}
  \caption{time $t_3$}
  \label{fig:3}
\end{subfigure}

\medskip

\begin{subfigure}{0.25\textwidth}
  \includegraphics[width=\linewidth]{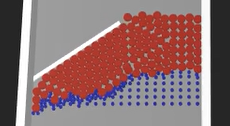}
  \label{fig:4}
\end{subfigure}\hfil 
\begin{subfigure}{0.25\textwidth}
  \includegraphics[width=\linewidth]{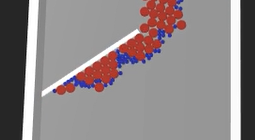}
  \label{fig:5}
\end{subfigure}\hfil 
\begin{subfigure}{0.25\textwidth}
  \includegraphics[width=\linewidth]{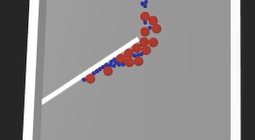}
  \label{fig:6}
\end{subfigure}
\caption{The larger particles (red) start first with the following smaller particles (blue) creeping through the mass of larger particles.} 
\label{Figure3}
\end{figure}

We also analyze the amount of mass that passes the obstacle, i.e., the outflow at the end of the conveyor belt for each particle class. The normalised time-dependent mass function describing the outflow to the solution of the conservation law is given by
\begin{align}
    U^c_{\rho}(t) &= \frac{1}{\int_{\Omega} \rho^c(\mathbf{x},0) \, \mathrm{d}\mathbf{x}} \, \int_{\Omega} \rho^c(\mathbf{x},t) \, \mathrm{d}\mathbf{x}.
\end{align}
where $\Omega$ is whole domain of the conveyor belt. The outflow curves for both test cases are shown in Figure~\ref{Figure4} and Figure~\ref{Figure5}, respectively.

\begin{figure}
    \begin{subfigure}{0.5\textwidth}
    \includegraphics[width=\linewidth]{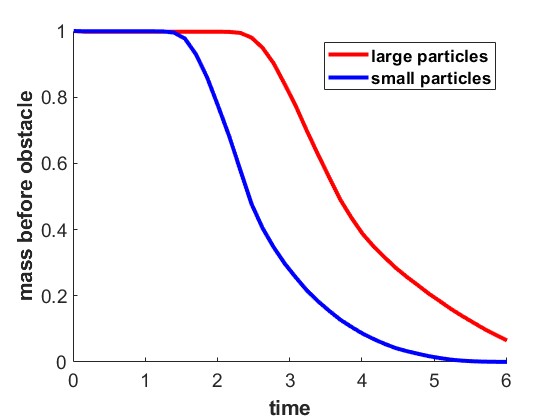}
    \caption{Small particles in front}
    \label{Figure4}
    \end{subfigure}\hfil 
    \begin{subfigure}{0.5\textwidth}
    \includegraphics[width=\linewidth]{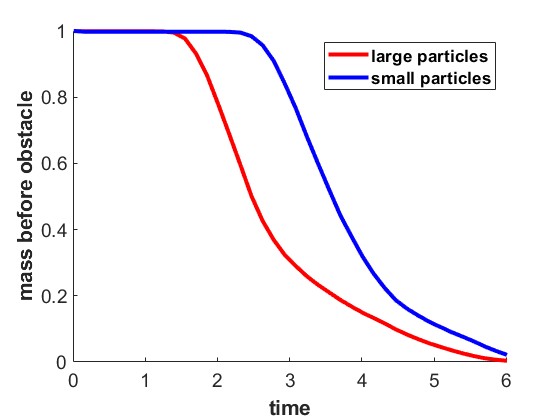}
    \caption{Large particles in front}
    \label{Figure5}
    \end{subfigure}\hfil 
    \caption{Outflow diagram for both test cases. By penetrating the mass of large particles, the mass of smaller particles catches up with the mass of the larger particles even if the large particles start first.}
\end{figure}

\section{Conclusion}

In this paper, we have proved the well-posedness of a nonlocal model describing heterogeneous flow on conveyor belts using the Roe scheme. Based on arguments borrowed from~\cite{Rossietal}, the results are generalized to the system case. The key challenge lies in the treatment of the discontinuity that occurs in the flux function. A regularized version of the Heaviside function $H$ is employed to establish BV bounds for the Roe-scheme. To extend these results to the system case, we consider a modified model incorporating an additional convolution in $H$. This approach enables us to derive estimates analogous to those obtained for the scalar model.  The numerical simulations show that the model captures effects for the mixing of different particle sizes, which can be observed in the physical simulations.

\section{Acknowledgements}

The project is funded by the Deutsche Forschungsgemeinschaft (DFG, German Research Foundation) - under the grant 526006304.

  \begin{appendices}
    \section{Technical Lemma}
    \label{sec:teclem}
    \begin{lemma}
      Let \ref{H} and \ref{eta} hold for all $c\in\{1,\dots,N\}$. Then,
      for $n=0, \ldots, N_T$, for $i, j \in \interi$, the following
      estimates hold:
      \begin{align}
        \label{eq:Jinf}
        \norma{J^{c,n}_\ell}_{\L\infty} \leq \
        & \epsilon_c \quad \mbox{ for } \ell=1,2,
        \\
        \label{eq:Jx}
        \modulo{J^{c,n}_1 (x_{i+1/2,j}) - J_1^{c,n} (x_{i-1/2,j})}
        \leq \
        & \left(2 \norma{\nabla^2 \eta_c}_{\L\infty}+ L_H\,\norma{\nabla \Tilde{\eta}_c}_{\L\infty}\right) \epsilon_c \, \dx \,   \norma{r_\Omega^{n}}_{\L1},
        \\
        \label{eq:J1y}
        \modulo{J_1^{c,n} (x_{i+1/2,j}) - J_1^{c,n}(x_{i+1/2,j+1})}
        \leq \
        & 
        \left(2 \norma{\nabla^2 \eta_c}_{\L\infty}+ L_H\,\norma{\nabla \Tilde{\eta}_c}_{\L\infty}\right) \epsilon_c \, \dy \,   \norma{r_\Omega^{n}}_{\L1},
        \\
        \nonumber
        \modulo{J^{c,n}_2 (x_{i,j+1/2}) - J_2^{c,n} (x_{i,j-1/2})}
        \leq \
        & \left(2 \norma{\nabla^2 \eta_c}_{\L\infty}+ L_H\,\norma{\nabla \Tilde{\eta}_c}_{\L\infty}\right) \epsilon_c \, \dy \,   \norma{r_\Omega^{n}}_{\L1},
        \\
        \nonumber
        \modulo{J^{c,n}_2 (x_{i+1,j+1/2}) - J_2^{c,n} (x_{i,j+1/2})}
        \leq \
        & \left(2 \norma{\nabla^2 \eta_c}_{\L\infty}+ L_H\,\norma{\nabla \Tilde{\eta}_c}_{\L\infty}\right) \epsilon_c \, \dx \,   \norma{r_\Omega^{n}}_{\L1},
        \end{align}
        and 
        \begin{align}
        \label{eq:Jtripla}
        \modulo{J^{c,n}_1 (x_{i+3/2,j}) -2 \,J^{c,n}_1 (x_{i+1/2,j}) + J_1^{c,n} (x_{i-1/2,j})}
        \leq \
        &   2 \,\epsilon_c \,  (\dx)^2 \left(
          c^c_1  \norma{r_\Omega^{n}}_{\L1} + c^c_2 \norma{r_\Omega^{n}}^2_{\L1}
          \right),
        \\
        \nonumber
        \modulo{J^{c,n}_2 (x_{i,j+3/2}) -2 \,J^{c,n}_2 (x_{i,j+1/2}) + J_2^{c,n} (x_{i,j-1/2})}
        \leq \
        &  2\, \epsilon_c \,  (\dy)^2 \left(
          c^c_1  \norma{r_\Omega^{n}}_{\L1} + c^c_2 \norma{r_\Omega^{n}}^2_{\L1}
          \right),
        \\
        \label{eq:Jtriplabis}
        \left|J_1^{c,n} (x_{i+1/2,j}) \!-\! J_1^{c,n}(x_{i+1/2,j+1})\! -\! J_1^{c,n} (x_{i-1/2,j}) \!+\!
        J_1^{c,n} \right.\! &\! \left.(x_{i-1/2,j+1})\right|
                          \leq \ 2 \, \epsilon_c \,  \dx \, \dy  \, C_c ,
        \\
        \nonumber
        \left|J_2^{c,n} (x_{i,j+1/2}) \!-\! J_2^{c,n}(x_{i+1,j+1/2})\! -\! J_2^{c,n} (x_{i,j-1/2}) \!+\!
        J_2^{c,n} \right. \! &\! \left.(x_{i+1,j-1/2})\right|
                           \leq \ 2 \, \epsilon_c \,  \dx \, \dy  \, C_c ,
      \end{align}
      where we set $r_\Omega^n= \sum_{k=1}^{N+1}\alpha_k\,\rho_\Omega^{k,n}$ and
      \begin{align}
        \label{eq:c12}
        & C_c = 
          c^c_1 \norma{r_\Omega^{n}}_{\L1} + c^c_2 \norma{r_\Omega^{n}}^2_{\L1}, 
        \\
        \nonumber
        &
          c^c_1 = \
         2 \, \norma{\nabla^3 \eta_c}_{\L\infty} 
         +\frac{1}{2}L_H\norma{\nabla^2 \tilde{\eta}_c}_{\L\infty},
         \\
         \nonumber
         &
          c^c_2 = \
         3 \, \norma{\nabla^2 \eta_c}^2_{\L\infty} 
         + 2\,L_H\,\norma{\nabla\Tilde{\eta}_c}_{\L\infty}\norma{\nabla^2\eta_c}_{\L\infty}
         +\frac{1}{2}L_{H'}
         \norma{\nabla \tilde{\eta}_c}_{\L\infty}^2
      \end{align}
    \end{lemma}
    \begin{proof}
      The proof of~\eqref{eq:Jinf} is immediate.

      Pass now to~\eqref{eq:Jx}. For the sake of simplicity, introduce
      the following notation:
      \begin{align*}
      \begin{split}
        D^c_+ =\
        & \sqrt{1 + \norma{\left(\nabla \eta_c * r_{\Omega}^n\right) (x_{i+1/2}, y_j)}^2},\\
         D^c_- =\
        & \sqrt{1 + \norma{\left(\nabla \eta_c * r_{\Omega}^n\right) (x_{i-1/2}, y_j)}^2}.
        \end{split}
      \end{align*}
      Hence, setting $r_{\Omega,h,\ell}^n = \sum_{k=1}^{N+1}\alpha_k\,\rho^{k,n}_{\Omega,h,\ell}$
      \begin{align}
        \label{eq:2original}
        & \modulo{J^{c,n}_1 (x_{i+1/2}, y_j) - J_1^{c,n}
          (x_{i-1/2},y_j)}
        \\
        \nonumber = \
        & \epsilon_c \,\Biggl| H_{i+1/2} \frac{\dx \, \dy}{D^c_+} \,
          \sum_{h, \ell \in \interi} r_{\Omega,h,\ell}^n \, \partial_1
          \eta_c(x_{i+1/2-h},y_{j-\ell}) 
          \\
        &  -
         H_{i-1/2}
         \frac{\dx \, \dy}{D^c_-} \,
          \sum_{h, \ell \in \interi} r_{\Omega,h,\ell}^n \, \partial_1
          \eta_c(x_{i-1/2-h},y_{j-\ell})\Biggr|
        \\ \label{eq:2} \leq \
        & \epsilon_c\,
          \modulo{ H_{i+1/2}} \frac{\dx \, \dy}{D^c_+} \, \modulo{\sum_{h, \ell \in \interi}
          r_{\Omega,h,\ell}^n \, \left(
          \partial_1 \eta_c(x_{i+1/2-h},y_{j-\ell}) -
          \partial_1 \eta_c(x_{i-1/2-h},y_{j-\ell}) \right) }
        \\\label{eq:2b}
        & + \epsilon_c \, H_{i+1/2}\, \dx \, \dy \,
          \modulo{\frac{1}{D^c_+} - \frac{1}{D^c_-}} \, \modulo{\sum_{h, \ell \in
          \interi} r_{\Omega,h,\ell}^n \partial_1
          \eta_c(x_{i-1/2-h},y_{j-\ell})}
        \\\label{eq:2c}
        & + \epsilon_c \, \modulo{H_{i+1/2}-H_{i-1/2}}\, \dx \, \dy \,
          \modulo{ \frac{1}{D^c_-}} \, \modulo{\sum_{h, \ell \in
          \interi} r_{\Omega,h,\ell}^n \partial_1
          \eta_c(x_{i-1/2-h},y_{j-\ell})}.
      \end{align}
      where we define for $\mu\in\{i-1/2,i+1/2,i+3/2\}$
      \begin{align*}
          H_\mu = H\left(\dx\dy \sum_{h,\ell\in\interi}r_{\Omega,h,\ell}^n \,\Tilde{\eta}_c(x_{\mu-h},y_{j-\ell}) - r_{max}\right)
      \end{align*}
      Consider~\eqref{eq:2}: since $D^c_+ \geq 1$, $H_\mu\leq1$ and
      \begin{displaymath}
        \modulo{ \partial_1 \eta_c(x_{i+1/2-h},y_{j-\ell})- \partial_1 \eta_c(x_{i-1/2-h},y_{j-\ell})}
        \leq
        \int_{x_{i-1/2-h}}^{x_{i+1/2-h}}\modulo{\partial_{11}^2 \eta_c (x,y_{j-\ell})} \d{x},
      \end{displaymath}
      we obtain    
      \begin{equation}
        \label{eq:2ok}
        [\eqref{eq:2}]
        \leq
        \epsilon_c \, \dx \, \norma{\partial_{11}^2 \eta_c}_{\L\infty} \norma{r_\Omega^{n}}_{\L1}.
      \end{equation}
      On the other hand, to estimate~\eqref{eq:2b}, compute
      \begin{displaymath}
        \modulo{\frac{1}{D^c_+} - \frac{1}{D^c_-}} = \frac{\modulo{D^c_+ - D^c_-}}{D^c_+ \, D^c_-}.
      \end{displaymath}
      Introduce $a_c(x) = \nabla \eta_c*r_\Omega^{n} (x)$ and
      $b(z) = (1+\norma{z}^2)^{1/2}$, for $z \in \reali^2$. In
      particular compute
      $ \nabla b (z) = \dfrac{z}{(1+\norma{z}^2)^{1/2}}$ and observe
      that $\|\nabla b(z)\|\leq 1$. Then
      \begin{align}
        \nonumber
        \modulo{D^c_+ - D^c_-} = \
        & \modulo{b(a_c(x_{i+1/2})) - b(a_c(x_{i-1/2}))}
          = \modulo{\nabla b(a_c(\tilde x_i)) \cdot \, \partial_xa_c(\tilde x_i) \, (x_{i+1/2} - x_{i-1/2})}
        \\
        \nonumber
        = \
        & \modulo{\frac{a_c (\tilde x_i)}{(1+\|a_c (\tilde x_i)\|^2)^{1/2}}\cdot \,
          \left(\partial_x \nabla \eta_c * r_\Omega^{n}\right)(\tilde x_i) \, \dx}
        \\
        \label{eq:5}
        \leq \
        &
          \dx \, \norma{r_\Omega^{n}}_{\L1} \norma{\nabla^2 \eta_c}_{\L\infty}.
      \end{align}
      Therefore,
      \begin{align}
        \label{eq:2bok}
        & \epsilon_c \,H_{i+1/2}\, \dx \, \dy \,
        \modulo{\frac{1}{D^c_+} - \frac{1}{D^c_-}} \, \modulo{\sum_{h, \ell \in
          \interi} r_{\Omega,h,\ell}^{n} \partial_1
          \eta_c(x_{i-1/2-h},y_{j-\ell})} 
          \\
        \nonumber \leq
        & \epsilon_c \, \dx \, \norma{\nabla^2 \eta_c}_{\L\infty} \norma{r_\Omega^{n}}_{\L1}.
      \end{align}
     Considering~\eqref{eq:2c}, by assumption \ref{H} we get
      \begin{align*}
          |H_{i+1/2}-H_{i-1/2}| & \leq L_H\!\! \left(\dx\dy\sum_{h,\ell\in\interi}r_{\Omega,h,\ell}^{n}\left(\Tilde{\eta}_c(x_{i+1/2-h},y_{j-\ell})-\Tilde{\eta}_c(x_{i-1/2-h},y_{j-\ell})\right)\right)
          \\
          &\leq L_H \dx \, \norma{\nabla \Tilde{\eta}_c}_{\L\infty} \norma{r_\Omega^{n}}_{\L1}
      \end{align*}
      Since $\dfrac{
     (\partial_1\eta_c * r_\Omega^{n}) (x_{i-1/2,j})}{D^c_{-}} \leq
   1$ we obtain therefore
      \begin{align}
      \label{eq:2cok}
         \nonumber & \epsilon_c \, |H_{i+1/2}-H_{i-1/2}|\, \dx \, \dy \,
          \modulo{ \frac{1}{D^c_-}} \, \modulo{\sum_{h, \ell \in
          \interi} r^{n}_{\Omega,h,\ell} \partial_1
          \eta_c(x_{i-1/2-h},y_{j-\ell})}
          \\ 
          \leq & \epsilon_c\,L_H\,\dx\,\norma{\nabla \Tilde{\eta}_c}_{\L\infty} \norma{r_\Omega^{n}}_{\L1}
      \end{align}

 \noindent Inserting~\eqref{eq:2ok},~\eqref{eq:2bok} and~\eqref{eq:2cok} into the estimate
   of~\eqref{eq:2original} yields the desired result.

\smallskip
   Consider now~\eqref{eq:Jtripla}. Introduce the following notation:
   for $\mu \in \left\{-1 ;\, 1;\, 3 \right\}$ set
   \begin{displaymath}
     D^c_{\mu} = \sqrt{1 + \norma{\left(\nabla\eta_c * r_\Omega^{n}\right) (x_{i+\mu/2}, y_j)}^2}.
   \end{displaymath}
   Thus
   \begin{align*}
     & J_1^{c,n} (\x{i+3/2,j}) - 2 \, J_1^{c,n} (\x{i+1/2,j}) + J_1^{c,n} (\x{i-1/2,j})
     \\
     = 
     & - \epsilon_c \left(\frac{H_{i+3/2}(\partial_1\eta_c * r_\Omega^{n}) (\x{i+3/2,j})}{D^c_3}
       - 2 \, \frac{H_{i+1/2}(\partial_1\eta_c *  r_\Omega^{n}) (\x{i+1/2,j})}{D^c_1}\right.
    \\
    & \quad
    \left.
       + \frac{H_{i-1/2}(\partial_1\eta_c * r_\Omega^{n}) (\x{i-1/2,j})}{D^c_{-1}} \pm \frac{H_{i+3/2}(\partial_1\eta_c * r_\Omega^{n}) (x_{i+3/2,j})}{D^c_1}\right.
     \\
     & \quad
     \left.
       \pm \frac{H_{i-1/2}(\partial_1\eta_c * r_\Omega^{n}) (x_{i-1/2,j})}{D^c_1}
        \pm \frac{H_{i+1/2}(\partial_1\eta_c * r_\Omega^{n}) (x_{i+3/2,j})}{D^c_1}\right.
     \\
     & \quad
     \left. 
     \pm \frac{H_{i+1/2}(\partial_1\eta_c * r_\Omega^{n}) (x_{i-1/2,j})}{D^c_1}
     \right)
     \\
     = 
     & - \epsilon_c \left(
       \left(\frac{1}{D^c_3} - \frac{1}{D^c_1}\right) H_{i+3/2}(\partial_1\eta_c * r_\Omega^{n}) (x_{i+3/2,j})\right.
       \\
    & \quad 
        + \left(H_{i+3/2}-H_{i+1/2}\right) \frac{1}{D_1^c}
        ( \partial_1\eta_c * r_\Omega^{n}) (x_{i+3/2,j})         
    \\
    & \quad
       \left.+ \frac{1}{D^c_1}
       \left(H_{i+1/2}(\partial_1\eta_c * r_\Omega^{n}) (x_{i+3/2,j})
         - H_{i+1/2}(\partial_1\eta_c * r_\Omega^{n}) (x_{i+1/2,j})\right)\right.
     \\
     & \quad \left.
         + \frac{1}{D^c_1}\!
       \left(H_{i+1/2}(\partial_1\eta_c * r_\Omega^{n}) (x_{i-1/2,j})
         - H_{i+1/2}(\partial_1\eta_c * r_\Omega^{n}) (x_{i+1/2,j})\right)\right.
       \\
    & \quad  +
        \left(H_{i-1/2}-H_{i+1/2}\right) \frac{1}{D_1^c}
        (\partial_1\eta_c * r_\Omega^{n}) (x_{i-1/2,j})   
       \\
    & \quad \left.
       +
       \left(\frac{1}{D^c_{-1}} - \frac{1}{D^c_1}\right) H_{i-1/2} (\partial_1\eta_c * r_\Omega^{n}) (x_{i-1/2,j})
     \right).
   \end{align*}
   Consider the terms separately, forgetting for a moment the
   $\epsilon_c$ in front of everything. Focus first on the terms with
   common denominator $D^c_1$ and $H_{i+1/2}$:
   \begin{align}
     \nonumber
     & \frac{ H_{i+1/2}}{D^c_1}    \left((\partial_1\eta_c * r_\Omega^{n}) (x_{i+3/2,j})\!
       - (\partial_1\eta_c * r_\Omega^{n}) (x_{i+1/2,j})
       \right.
     \\  \nonumber
    & \quad 
    \left. + (\partial_1\eta_c * r_\Omega^{n}) (x_{i-1/2,j})
       - (\partial_1\eta_c * r_\Omega^{n}) (x_{i+1/2,j})\right)
     \\  \nonumber
     = \
     & H_{i+1/2}\frac{\dx \, \dy}{D^c_1} \sum_{h, \ell \in \interi} r_{\Omega,h, \ell}^n \,
     \left(
       \partial_1 \eta_c (x_{i+3/2 -h}, y_{j-\ell})
       -  \partial_1 \eta_c (x_{i+1/2 -h}, y_{j-\ell})
     \right.
     \\  \nonumber
     & \qquad\qquad\qquad\qquad\quad
     \left.
       +  \partial_1 \eta_c (x_{i-1/2 -h}, y_{j-\ell})
       -  \partial_1 \eta_c (x_{i+1/2 -h}, y_{j-\ell})
     \right)
     \\  \nonumber
     = \ 
     & H_{i+1/2}
     \frac{\dx \, \dy}{D^c_1} \sum_{h, \ell \in \interi} r_{\Omega,h, \ell}^n \, \dx \,
     \left( \partial_{11}^2 \eta_c(\hat x_{i+1-h}, y_{j-\ell})
       -  \partial_{11}^2 \eta_c(\hat x_{i-h}, y_{j-\ell}) \right)
     \\  \nonumber
     = \
     & H_{i+1/2}     \frac{\dx \, \dy}{D^c_1} \sum_{h, \ell \in \interi} r_{\Omega,h, \ell}^n \, \dx \,
     \int_{\hat x_{i-h}}^{\hat x_{i+1-h}} \partial_{111}^3 \eta_c (x, y_{j-\ell}) \d{x}
     \\ \label{eq:7}
     \leq \
     & 2 \, (\dx)^2 \, \norma{\partial_{111}^3 \eta_c}_{\L\infty} \norma{r_\Omega^{n}}_{\L1},
   \end{align}
   with $\hat x_{i-h} \in \, ]x_{i-1/2-h}, x_{i+1/2-h}[$. 
   Next, we obtain 
   for $r^n_\Omega = \sum_{k=1}^{N+1}\alpha_k\rho_\Omega^{k,n}$ 
   \begin{align}
   \label{eq:4dnew}
       \nonumber 
       & \left(H_{i+3/2}-H_{i+1/2}\right) \frac{1}{D_1^c}
         (\partial_1\eta_c * r_\Omega^{n}) (x_{i+3/2,j}) 
         + 
         \left(H_{i-1/2}-H_{i+1/2}\right) \frac{1}{D_1^c}
         (\partial_1\eta_c * r_\Omega^{n}) (x_{i-1/2,j})
         \\
         \nonumber 
         =&
          \left(H_{i+3/2}-2H_{i+1/2}+H_{i-1/2}\right)
          \frac{1}{D_1^c}
         (\partial_1\eta_c * r_\Omega^{n}) (x_{i+1/2,j}) 
          \\\nonumber 
          &
          +
          \left(H_{i+3/2}-H_{i+1/2}\right) \frac{1}{D_1^c}
          \left((\partial_1\eta_c * r_\Omega^{n}) (x_{i+3/2,j})-
          (\partial_1\eta_c * r_\Omega^{n}) (x_{i+1/2,j})\right)
          \\
          &
          +
          \left(H_{i-1/2}-H_{i+1/2}\right) \frac{1}{D_1^c}
          \left((\partial_1\eta_c * r_\Omega^{n}) (x_{i-1/2,j})-
          (\partial_1\eta_c * r_\Omega^{n}) (x_{i+1/2,j})\right)
       .
    \end{align}
    Since $|(\partial_1\eta_c*r^n_\Omega)(x_{i+1/2,j})| \le\|(\nabla\eta_c*r^n_\Omega)(x_{i+1/2,j})\|\le D^c_1$,
    we can bound the first term by
    \begin{align*}
        |H_{i+3/2}-2H_{i+1/2}+H_{i-1/2}|\le(\Delta x)^2\left(L_{H'}\|\nabla\tilde\eta_c\|^2_{L^\infty}\|r^n_\Omega\|^2_{L^1}
        +L_H\|\nabla^2\tilde\eta_c\|_{L^\infty}\|r^n_\Omega\|_{L^1}\right).
    \end{align*}
    For the second and the third term, we have
    \begin{align*}
        |H_{i+3/2}-H_{i+1/2}|, |H_{i-1/2}-H_{i+1/2}|
         \le\Delta x\,L_H\|\nabla\tilde\eta_c\|_{L^\infty}\|r^n_\Omega\|_{L^1},
    \end{align*}
    and, since $D^c_1\ge1$,
    \begin{align*}
        \frac{1}{D^c_1}|(\partial_1\eta_c*r^n_\Omega)(x_{i\pm1+1/2,j})-(\partial_1\eta_c*r^n_\Omega)(x_{i+1/2,j})|
        &\le\int_{x_{i+1/2}}^{x_{i+1/2}\pm\Delta x}\big|\partial^2_{11}\eta_c*r^n_\Omega\big|\,\mathrm{d}x
        \\
        &\le\Delta x\,\|\nabla^2\eta_c\|_{L^\infty}\|r^n_\Omega\|_{L^1}.
    \end{align*}
    Multiplying and summing up the two parts yields
    \begin{align*}
        2(\Delta x)^2L_H\|\nabla\tilde\eta_c\|_{L^\infty}\|\nabla^2\eta_c\|_{L^\infty}\|r^n_\Omega\|^2_{L^1}.
    \end{align*}
Combining everything we get 
\begin{align}
\begin{split}
|\ref{eq:4dnew} |
&\le (\Delta x)^2 \Big( L_{H'} \|\nabla\tilde{\eta}_c\|^2_{L^\infty} \|r_\Omega^n\|^2_{L^1} + L_H \|\nabla^2\tilde{\eta}_c\|_{L^\infty} \|r_\Omega^n\|_{L^1} \\
&\quad + 2 L_H \|\nabla\tilde{\eta}_c\|_{L^\infty} \|\nabla^2\eta_c\|_{L^\infty} \|r_\Omega^n\|^2_{L^1} \Big).
\end{split} \label{eq:4d}
\end{align}
    We are left with 
   \begin{align}
     \label{eq:4}
      &  H_{i+3/2}\left(\frac{1}{D^c_3} - \frac{1}{D^c_1}\right) \left(\partial_1\eta_c * r^n_\Omega\right) (x_{i+3/2,j})
     + 
      H_{i-1/2}
      \left(\frac{1}{D^c_{-1}} - \frac{1}{D^c_1}\right) \left(\partial_1\eta_c * r^n_\Omega\right) (x_{i-1/2,j}).
   \end{align}
   Add and subtract to~\eqref{eq:4}
   \begin{displaymath}
     H_{i+3/2}\left(\frac{1}{D^c_{-1}} - \frac{1}{D^c_1}\right) \left(\partial_1\eta_c * r^n_\Omega\right) (x_{i+3/2,j}).
   \end{displaymath}
   and 
   \begin{align*}
       H_{i+3/2}\left(\frac{1}{D^c_{-1}} - \frac{1}{D^c_1}\right) \left(\partial_1\eta_c * r^n_\Omega\right) (x_{i-1/2,j}).   
    \end{align*}
   Hence,
   \begin{align}
     \label{eq:4a}
    &  \left(\frac{1}{D^c_3} - 2 \, \frac{1}{D^c_1} + \frac{1}{D^c_{-1}} \right) H_{i+3/2}
     \left(\partial_1\eta_c * r^n_\Omega\right) (x_{i+3/2,j})
    \\
    \label{eq:4b}
    & \quad+ H_{i+3/2}
     \left(\frac{1}{D^c_{-1}} - \frac{1}{D^c_1}\right) \left(
       \left(\partial_1\eta_c * r^n_\Omega\right) (x_{i-1/2,j})
       - \left(\partial_1\eta_c * r^n_\Omega\right) (x_{i+3/2,j})
     \right).
     \\
     \label{eq:4c}
     & \quad +\left(H_{i-1/2}-H_{i+3/2}\right)
     \left(\frac{1}{D^c_{-1}} - \frac{1}{D^c_1}\right)
     \left(\partial_1\eta_c * r^n_\Omega\right) (x_{i-1/2,j})
   \end{align}
   Consider first~\eqref{eq:4b}: exploiting also~\eqref{eq:5}, we obtain
   \begin{align}
     \nonumber
     [\eqref{eq:4b}]= \
     & H_{i+3/2}
     \frac{D^c_1 - D^c_{-1}}{D^c_1 \, D^c_{-1}} \, \dx \, \dy
     \\
     \nonumber
     & \quad\quad
     \sum_{h, \ell \in \interi} r^{n}_{\Omega,h,\ell}\left(
       \partial_1 \eta_c(x_{i-1/2-h}, y_{j-\ell})
       -
       \partial_1 \eta_c(x_{i+3/2-h}, y_{j-\ell})
     \right)
     \\ \nonumber
     = \
     & H_{i+3/2}
     \frac{D^c_1 - D^c_{-1}}{D^c_1 \, D^c_{-1}} \, \dx \, \dy
     \sum_{h, \ell \in \interi} r^{n}_{\Omega,h,\ell}
     \int_{x_{i+3/2-h}}^{x_{i-1/2-h}} \partial_{11}^2 \eta_c(x,y_{j-\ell}) \d{x}
     \\ \label{eq:4bok}
     \leq \
     &
     2 \, (\dx)^2 \, \norma{\nabla^2 \eta_c}^2_{\L\infty} \norma{r_\Omega^{n}}^2_{\L1}.
   \end{align}
   As far as~\eqref{eq:4a} is concerned, focus on the terms in the brackets:
   \begin{align}
     \nonumber
     \frac{1}{D^c_3} - 2 \, \frac{1}{D^c_1} + \frac{1}{D^c_{-1}}
     = \
     & \frac{D^c_1 \, D^c_{-1} - 2 \, D^c_3 \, D^c_{-1} + D^c_3 \, D^c_1}{D^c_3 \, D^c_1 \, D^c_{-1}}
     \\ \nonumber
     = \
     & \frac{2\,(D^c_1-D^c_3)(D^c_{-1}-D^c_1)-D^c_1\big(D^c_{-1}-2\,D^c_1+D^c_3\big)}
         {D^c_3\,D^c_1\,D^c_{-1}}
     \\  \label{eq:6}
     = \
     &  \frac{2\,(D^c_1-D^c_3)(D^c_{-1}-D^c_1)}{D^c_3\,D^c_1\,D^c_{-1}}
   -\frac{D^c_{-1}-2\,D^c_1+D^c_3}{D^c_3\,D^c_{-1}}.
   \end{align}
   Inserting the first addend of~\eqref{eq:6} back into~\eqref{eq:4a} yields
   \begin{equation}
     \label{eq:4a1ok}
     \frac{2(D^c_{1} - D^c_3)(D^c_{-1}-D^c_1)}{D^c_3 \, D^c_1 \, D^c_{-1}}
     H_{i+3/2}
     \left(\partial_1\eta_c * r^n_\Omega\right) (x_{i+3/2,j})
     \leq
     2 \, (\dx)^2 \norma{\nabla^2 \eta_c}_{\L\infty}^2 \norma{r_\Omega^{n}}_{\L1}^2,
   \end{equation}
   where we exploit~\eqref{eq:5} twice and use the fact that $\dfrac{
     (\partial_1\eta_c * r_\Omega^{n}) (x_{i+3/2,j})}{D^c_3} \leq
   1$. Concerning the second addend of~\eqref{eq:6}, focus on its
   numerator: with the notation introduced before~\eqref{eq:5},
   \begin{align*}
     D^c_{-1}\! - 2 \, D^c_1 + D^c_3 = \
     & b\left(a_c(x_{i+3/2})\right) - 2 \, b\left(a_c(x_{i+1/2})\right)
     + b\left(a_c(x_{i-1/2})\right)
     \\
     = \
     & \nabla b \left(a_c(\check x_{i+1})\right) \, \partial_x a_c (\check x_{i+1}) (x_{i+3/2} - x_{i+1/2})
     -  \nabla b \left(a_c(\check x_{i})\right) \, \partial_x a_c (\check x_{i}) (x_{i+1/2} - x_{i-1/2})
     \\
     = \
     & \dx \left(
       \nabla b \left(a_c(\check x_{i+1})\right) (\partial_1 \nabla \eta_c * r^n_\Omega) (\check x_{i+1})
       -  \nabla b \left(a_c(\check x_{i})\right)  (\partial_1 \nabla \eta_c * r^n_\Omega) (\check x_{i})
     \right)
     \\
     & \pm \dx \,\, \nabla b\!\left(a_c (\check x_i)\right)
     (\partial_1 \nabla \eta_c * r^n_\Omega) (\check x_{i+1})
     \\
     = \
     & \dx \, \left[
      \nabla b \left(a_c(\check x_{i+1})\right) - \nabla b\left(a_c (\check x_i)\right)
     \right] (\partial_1 \nabla \eta_c * r^n_\Omega) (\check x_{i+1})
     \\
     & + \dx \, \, \nabla b\!\left(a_c (\check x_i)\right) \left[
        (\partial_1 \nabla \eta_c * r^n_\Omega) (\check x_{i+1})
        - (\partial_1 \nabla \eta_c * r^n_\Omega) (\check x_{i})
     \right]
   \end{align*}
   where $\check x_i \in \, ]x_{i-1/2}, x_{i+1/2}[$  and
$\check x_{i+1}\in\left]x_{i+1/2},x_{i+3/2}\right[$. Since $\|\nabla^2 b(z)\|\le 1$ for all $z\in\mathbb{R}^2$, the map $\nabla b$ is Lipschitz
continuous with constant $1$, and $|\check x_{i+1}-\check x_i|\le 2\Delta x$. We then get 
   \begin{align*}
   \modulo{D^c_{-1}-2D^c_1+D^c_3}
    &\le \Delta x\,\|\nabla b(a_c(\check x_{i+1}))-\nabla b(a_c(\check x_i))\|\,
       \|(\partial_1\nabla\eta_c*r^n_\Omega)(\check x_{i+1})\| \\
    &\quad+\Delta x\,\|\nabla b(a_c(\check x_i))\|\,
       \|(\partial_1\nabla\eta_c*r^n_\Omega)(\check x_{i+1})
            -(\partial_1\nabla\eta_c*r^n_\Omega)(\check x_i)\| \\
    &\le 2\,(\Delta x)^2\,\|\partial_1\nabla\eta_c\|^2_{\mathbf{L}^\infty}\|r^n_\Omega\|^2_{\mathbf{L}^1}
       +2\,(\Delta x)^2\,\|\partial^2_{11}\nabla\eta_c\|_{\mathbf{L}^\infty}\|r^n_\Omega\|_{\mathbf{L}^1}.
    \end{align*}
   Now insert this
   estimate back into~\eqref{eq:6} and~\eqref{eq:4a}: since $|(\partial_1 \eta_c *r^n_\Omega)(x_{i+3/2,j})|/D^c_3\leq1$ and $D^c_{-1}\geq 1$,
   \begin{equation}
   \begin{split}
     \label{eq:4a2ok}
     &\modulo{ \frac{D^c_{-1} - 2 \, D^c_1 + D^c_3}{D^c_3 \, D^c_{-1}}
      \left(\partial_1\eta_c * r_\Omega^{n}\right) (x_{i+3/2,j})}\\
    &\quad\leq 
    2 \, (\dx)^2 \left[
      \norma{\partial_1 \nabla \eta_c}_{\L\infty}^2 \norma{r_\Omega^{n}}_{\L1}^2
    + \norma{\partial_{11}^2 \nabla \eta_c}_{\L\infty} \norma{r_\Omega^{n}}_{\L1}\right].
    \end{split}
   \end{equation}
    Lastly, we consider~\eqref{eq:4c}. As before, by assumption~\eqref{H} we get 
    \begin{align*}
        |H_{i+3/2}-H_{i-1/2} |
        & \leq L_H\left(\Tilde{\eta}_c(x_{i+3/2j})*r_\Omega^{n} - \Tilde{\eta}_c(x_{i-1/2,j})*r_\Omega^{n}\right)
        \\
        &\leq
        L_H \, 2\dx\, \norma{\nabla \Tilde{\eta}_c}_{\L\infty}
        \norma{r_\Omega^{n}}_{\L1}
        \end{align*}
   and with~\eqref{eq:2bok} follows 
   \begin{align}
   \label{eq:4a3ok}
   \nonumber
   &|H_{i+3/2}-H_{i-1/2}|
     \modulo{\frac{1}{D^c_{-1}} - \frac{1}{D^c_1}}
     \left(\partial_1\eta_c * r_\Omega^{n}\right) (x_{i-1/2,j})
     \\
    &\quad \leq L_H\, 2(\dx)^2 \, \norma{\nabla\Tilde{\eta}_c}_{\L\infty}\norma{\nabla^2 \eta_c}_{\L\infty} \norma{r_\Omega^{n}}_{\L1}^2.
    \end{align}
   Collecting together~\eqref{eq:7}, \eqref{eq:4d}, \eqref{eq:4bok}, \eqref{eq:4a1ok}, ~\eqref{eq:4a2ok} and~\eqref{eq:4a3ok} yields
   \begin{align*}
&|J_1^{c,n}(x_{i+3/2, j}) - 2J_1^{c,n}(x_{i+1/2, j}) + J_1^{c,n}(x_{i-1/2, j})| \\
&\le 2 \epsilon_c (\Delta x)^2 \Bigg( \left( 2\|\nabla^3 \eta_c\|_{L^\infty} + \frac{1}{2} L_H\|\nabla^2 \tilde{\eta}_c\|_{L^\infty} \right) \|r_\Omega^n\|_{L^1} \\
&\quad + \left( 3\|\nabla^2 \eta_c\|_{L^\infty}^2 + 2 L_H\|\nabla \tilde{\eta}_c\|_{L^\infty} \|\nabla^2 \eta_c\|_{L^\infty} + \frac{1}{2} L_{H'}\|\nabla \tilde{\eta}_c\|^2_{L^\infty} \right) \|r_\Omega^n\|_{L^1}^2 \Bigg).
\end{align*}
{The proofs of the other inequalities follow analogously.}

 \end{proof}
\end{appendices}

\small{ \bibliography{matflow}

@article {ACG2015,
    AUTHOR = {Aggarwal, Aekta and Colombo, Rinaldo M. and Goatin, Paola},
     TITLE = {Nonlocal systems of conservation laws in several space
              dimensions},
   JOURNAL = {SIAM J. Numer. Anal.},
  FJOURNAL = {SIAM Journal on Numerical Analysis},
    VOLUME = {53},
      YEAR = {2015},
    NUMBER = {2},
     PAGES = {963--983},
      ISSN = {0036-1429},
   MRCLASS = {65M06 (35L65 45K05 65M12)},
  MRNUMBER = {3332915},
MRREVIEWER = {Zhongqiang Zhang},
       DOI = {10.1137/140975255},
       URL = {https://doi.org/10.1137/140975255},
}

@article {ACT2015,
    AUTHOR = {Amorim, Paulo and Colombo, Rinaldo M. and Teixeira, Andreia},
     TITLE = {On the numerical integration of scalar nonlocal conservation
              laws},
   JOURNAL = {ESAIM Math. Model. Numer. Anal.},
  FJOURNAL = {ESAIM. Mathematical Modelling and Numerical Analysis},
    VOLUME = {49},
      YEAR = {2015},
    NUMBER = {1},
     PAGES = {19--37},
      ISSN = {0764-583X},
   MRCLASS = {65M06 (65M12)},
  MRNUMBER = {3342191},
MRREVIEWER = {Prabir K. Daripa},
       DOI = {10.1051/m2an/2014023},
       URL = {https://doi.org/10.1051/m2an/2014023},
}

@article {BlandinGoatin2016,
    AUTHOR = {Blandin, Sebastien and Goatin, Paola},
     TITLE = {Well-posedness of a conservation law with non-local flux
              arising in traffic flow modeling},
   JOURNAL = {Numer. Math.},
  FJOURNAL = {Numerische Mathematik},
    VOLUME = {132},
      YEAR = {2016},
    NUMBER = {2},
     PAGES = {217--241},
      ISSN = {0029-599X},
   MRCLASS = {35L65 (35B30 35D30 35L45 35R09 76B03)},
  MRNUMBER = {3447130},
MRREVIEWER = {J\"{o}rg M. H\"{a}rterich},
       DOI = {10.1007/s00211-015-0717-6},
       URL = {https://doi.org/10.1007/s00211-015-0717-6},
}

@article {ColomboGaravelloMercier2012,
    AUTHOR = {Colombo, Rinaldo M. and Garavello, Mauro and L\'{e}cureux-Mercier,
              Magali},
     TITLE = {A class of nonlocal models for pedestrian traffic},
   JOURNAL = {Math. Models Methods Appl. Sci.},
  FJOURNAL = {Mathematical Models and Methods in Applied Sciences},
    VOLUME = {22},
      YEAR = {2012},
    NUMBER = {4},
     PAGES = {1150023, 34},
      ISSN = {0218-2025},
   MRCLASS = {90B20 (35Q91)},
  MRNUMBER = {2902155},
       DOI = {10.1142/S0218202511500230},
       URL = {https://doi.org/10.1142/S0218202511500230},
}

@article {KP2017,
    AUTHOR = {Keimer, Alexander and Pflug, Lukas},
     TITLE = {Existence, uniqueness and regularity results on nonlocal
              balance laws},
   JOURNAL = {J. Differential Equations},
  FJOURNAL = {Journal of Differential Equations},
    VOLUME = {263},
      YEAR = {2017},
    NUMBER = {7},
     PAGES = {4023--4069},
      ISSN = {0022-0396},
   MRCLASS = {35L65 (35B65 35D30 35L03)},
  MRNUMBER = {3670045},
MRREVIEWER = {Alberto Valli},
       DOI = {10.1016/j.jde.2017.05.015},
       URL = {https://doi.org/10.1016/j.jde.2017.05.015},
}

@article {KPS2018,
    AUTHOR = {Keimer, Alexander and Pflug, Lukas and Spinola, Michele},
     TITLE = {Existence, uniqueness and regularity of multi-dimensional
              nonlocal balance laws with damping},
   JOURNAL = {J. Math. Anal. Appl.},
  FJOURNAL = {Journal of Mathematical Analysis and Applications},
    VOLUME = {466},
      YEAR = {2018},
    NUMBER = {1},
     PAGES = {18--55},
      ISSN = {0022-247X},
   MRCLASS = {35L60 (35B65)},
  MRNUMBER = {3818104},
MRREVIEWER = {Stephen D. Pankavich},
       DOI = {10.1016/j.jmaa.2018.05.013},
       URL = {https://doi.org/10.1016/j.jmaa.2018.05.013},
}

@article {Kruzkov,
  AUTHOR =	 {Kru{\v{z}}hkov, Stanislav Nikolaevich},
  TITLE =	 {First order quasilinear equations with several
                  independent variables. },
  JOURNAL =	 {Mat. Sb. (N.S.)},
  VOLUME =	 {81 (123)},
  YEAR =	 1970,
  PAGES =	 {228--255},
  MRCLASS =	 {35.37},
  MRNUMBER =	 {42 \#2159},
  MRREVIEWER =	 {Z. Ziele{\'z}ny},
}

@article {original,
  AUTHOR =	 {Simone Göttlich and Simon Hoher and Patrick Schindler and Veronika Schleper and Alexander Verl},
  TITLE =	 {Modeling, simulation and validation of material flow on conveyor belts},
  JOURNAL =	 {Appl. Math. Model.},
  FJOURNAL = {Applied Mathematical Modelling},
  VOLUME = {38},
      YEAR = {2014},
    NUMBER = {13},
  PAGES =	 {3295--3313},
  ISSN = {0307-904X},
  DOI = {https://doi.org/10.1016/j.apm.2013.11.039},
}

@article {Rossietal,
    AUTHOR = {Rossi, Elena and Wei{\ss}en, Jennifer and Goatin, Paola and
              G\"ottlich, Simone},
     TITLE = {Well-posedness of a non-local model for material flow on
              conveyor belts},
   JOURNAL = {ESAIM Math. Model. Numer. Anal.},
  FJOURNAL = {ESAIM. Mathematical Modelling and Numerical Analysis},
    VOLUME = {54},
      YEAR = {2020},
    NUMBER = {2},
     PAGES = {679--704},
      ISSN = {2822-7840,2804-7214},
   MRCLASS = {65M08 (35B30 35L45 35L65 65M12)},
  MRNUMBER = {4074000},
MRREVIEWER = {Jinjie\ Liu},
       DOI = {10.1051/m2an/2019062},
       URL = {https://doi.org/10.1051/m2an/2019062},
}

@article{ GoatinRossi2024,
	author = {Goatin, Paola and Rossi, Elena},
	title = {Well-posedness of nonlocal macroscopic models of multi-population
pedestrian flows for domain shape optimization},
journal = {Journal of Hyperbolic Differential Equations},
volume = {23},
number = {01},
pages = {107-126},
year = {2026},
doi = {10.1142/S0219891626400059},

URL = { 
    
        https://doi.org/10.1142/S0219891626400059
    
    

},
eprint = { 
    
        https://doi.org/10.1142/S0219891626400059
    
}
}

@article {burger2020,
    AUTHOR = {B\"urger, Raimund and Goatin, Paola and Inzunza, Daniel and
              Villada, Luis Miguel},
     TITLE = {A non-local pedestrian flow model accounting for anisotropic
              interactions and domain boundaries},
   JOURNAL = {Math. Biosci. Eng.},
  FJOURNAL = {Mathematical Biosciences and Engineering. MBE},
    VOLUME = {17},
      YEAR = {2020},
    NUMBER = {5},
     PAGES = {5883--5906},
      ISSN = {1547-1063,1551-0018},
   MRCLASS = {35R09 (35B30 65M06 76A30 90B20)},
  MRNUMBER = {4160250},
       DOI = {10.3934/mbe},
       URL = {https://doi.org/10.3934/mbe},
}

@article {goatin2023,
    AUTHOR = {Goatin, Paola and Inzunza, Daniel and Villada, Luis Miguel},
     TITLE = {Nonlocal macroscopic models of multi-population pedestrian
              flows for walking facilities optimization},
   JOURNAL = {Appl. Math. Model.},
  FJOURNAL = {Applied Mathematical Modelling. Simulation and Computation for
              Engineering and Environmental Systems},
    VOLUME = {141},
      YEAR = {2025},
     PAGES = {Paper No. 115927, 14},
      ISSN = {0307-904X,1872-8480},
   MRCLASS = {65M06},
  MRNUMBER = {4851781},
       DOI = {10.1016/j.apm.2025.115927},
       URL = {https://doi.org/10.1016/j.apm.2025.115927},
}

@article{van2008fastlane,
  title={Fastlane: New multiclass first-order traffic flow model},
  author={Van Lint, JWC and Hoogendoorn, Serge P and Schreuder, Marco},
  journal={Transportation Research Record},
  volume={2088},
  number={1},
  pages={177--187},
  year={2008},
  publisher={SAGE Publications Sage CA: Los Angeles, CA}
}

@article{CrippaMercier2011,
author = {Crippa, Gianluca and Lécureux-Mercier, Magali},
year = {2011},
month = {12},
pages = {},
title = {Existence and uniqueness of measure solutions for a system of continuity
equations with non-local flow},
volume = {20},
journal = {Nonlinear Differential Equations and Applications NoDEA},
doi = {10.1007/s00030-012-0164-3}
}

  \bibliographystyle{abbrv} }

\end{document}